\documentclass[reqno, a4paper, 10pt]{amsart}

\usepackage{amsmath,amssymb,latexsym, amsfonts}
\usepackage{verbatim}
\usepackage{times}
\usepackage{hyperref}
\usepackage{latexsym}
\usepackage{graphics}
\usepackage[2emode]{psfrag}
\usepackage{mathrsfs}
\usepackage{stmaryrd}
\usepackage{amsthm}
\usepackage{amsmath}
\usepackage[all]{xy}
\usepackage{enumerate}
\usepackage[latin1]{inputenc}
\usepackage{lscape}
\usepackage{a4wide}
\usepackage[usenames]{color}
\usepackage{cancel}
\usepackage[normalem]{ulem}

\newcommand{\Der}{\mathrm{Der}}
\numberwithin{equation}{section}

\theoremstyle{plain}

\newtheorem{theorem}{Theorem}[section]
\newtheorem{thm}[theorem]{Theorem}
\newtheorem{proposition}[theorem]{Proposition}

\newtheorem{lemma}[theorem]{Lemma}
\newtheorem{lem}[theorem]{Lemma}

\newtheorem{cor}[theorem]{Corollary}

\theoremstyle{definition}
\newtheorem{definition}[theorem]{Definition}
\newtheorem{dfn}[theorem]{Definition}

\theoremstyle{remark}
\newtheorem{example}[theorem]{Example}

\newtheorem{rem}{Remark}

\newcommand{\emme}{{\scriptscriptstyle{M}}}

\newcommand{\erre}{{\scriptscriptstyle{R}}}
\newcommand{\esse}{{\scriptscriptstyle{S}}}
\newcommand{\uhhu}{{\scriptscriptstyle{U}}}

\newcommand{\C}{{\mathbb{C}}}

\newcommand{\kk}{{\Bbbk}}

\newcommand{\gb}{\beta}  

\newcommand{\gd}{\delta} 
\newcommand{\gD}{\Delta}
\newcommand{\Deltaell}{\Delta} 
\newcommand{\gve}{\varepsilon}

\newcommand{\gs}{\sigma}

\newcommand{\cC}{{\mathcal C}}

\newcommand{\cE}{{\mathcal E}}

\newcommand{\cG}{{\mathcal G}}

\newcommand{\cJ}{{\mathcal J}}

\newcommand{\cM}{{\mathcal M}}
\newcommand{\cN}{{\mathcal N}}

\newcommand{\cP}{{\mathcal P}}

\newcommand{\cS}{{\mathcal S}}

\newcommand{\cV}{{\mathcal V}}
\newcommand{\cW}{{\mathcal W}}

\newcommand{\End}[1]{\operatorname{End}(#1)}

\newcommand{\Hom}{\operatorname{Hom}}

\newcommand{\Tor}{{\rm Tor}}

\newcommand{\id}{{\rm id}}

\newcommand{\sll}{s}
\newcommand{\tl}{t}

\newcommand{\due}[3]{{}_{{#2 \!}} {#1}_{{\! #3}}\,}    

\newcommand{\pl}{\partial}

\newcommand{\rmref}[1]{{\rm (}\ref{#1}{\rm )}}

\newcommand{{\Hl}}{{H^{\ell}}} 
\newcommand{{\mHop}}{{m_{H^{\rm op}}}} 
\newcommand{{\Hop}}{{H^{\rm op}}} 
\newcommand{{\mUop}}{{m_{U^{\rm op}}}} 
\newcommand{{\Uop}}{{U^{\rm op}}}
\newcommand{{\mVop}}{{m_{V^{\rm op}}}} 
\newcommand{{\Vop}}{{V^{\rm op}}}  
\newcommand{{\Ae}}{{A^{\rm e}}}
\newcommand{{\Be}}{{B^{\rm e}}}
\newcommand{{\Aop}}{{A^{\rm op}}}
\newcommand{{\Aope}}{({A^{\rm op}})^{\rm e}}
\newcommand{{\Aopl}}{{A^{\rm op}_\pl}}

\newcommand{{\Bop}}{{B^{\rm op}}}
\newcommand{{\Bope}}{({B^{\rm op}})^{\rm e}}
\newcommand{{\Bpl}}{{B_\pl}}

\newcommand{{\Reh}}{{R^{\rm e}}}
\newcommand{{\Se}}{{S^{\rm e}}}
\newcommand{{\Rop}}{{R^{o}}}
\newcommand{{\Sop}}{{S^{o}}}

\newcommand{\Ropp}{{\scriptscriptstyle{\Rop}}}
\newcommand{\Sopp}{{\scriptscriptstyle{\Sop}}}
\newcommand{\Ree}{{\scriptscriptstyle{\Reh}}}
\newcommand{\See}{{\scriptscriptstyle{\Se}}}

\newcommand{{\Pe}}{{P^{\rm e}}}
\newcommand{{\Qe}}{{Q^{\rm e}}}

\newcommand{{\op}}{{{o}}}
\newcommand{{\coop}}{{{\rm coop}}}
\newcommand{{\sop}}{{*^{\rm op}}}

\newcommand{\Mod}{\mathbf{Mod}}                     %
\newcommand{\lact}{{\,\raise1pt\hbox{$\scriptscriptstyle{\rhd}$} \, }}                  %
\newcommand{\ract}{{\,\raise1pt\hbox{$\scriptscriptstyle{\lhd}$} \, }}                  
\newcommand{\blact}{{\,\raise1pt\hbox{$\scriptscriptstyle{\blacktriangleright}$} \, }}  %
\newcommand{\bract}{{\,\raise1pt\hbox{$\scriptscriptstyle{\blacktriangleleft}$} \, }}   %

\newcommand{{\gog}}{{G \rightrightarrows G_0}}

\newcommand{{\rra}}{\rightrightarrows}

\newcommand{{\lra}}{ \longrightarrow  }
\newcommand{{\lla}}{ \longleftarrow }
\newcommand{{\lma}}{ \longmapsto  }

\newcommand{{\bull}}{{\scriptscriptstyle{\bullet}}}

\newcommand{{\qqquad}}{{\quad\quad\quad}}

\input{xy}
 \xyoption{all}
\newcommand{\lr}[1]{\big( #1 \big)}
\newcommand{\LR}[1]{\left[\underset{}{} #1 \right]}
\newcommand{\tensor}[1]{\otimes_{\scriptscriptstyle{#1}}}

\newcommand{\Scr}[1]{\mathscr{#1}}
\newcommand{\Rings}[1]{#1\text{-}\mathbf{Rings}}
\newcommand{\td}[1]{\tilde{#1}}
\newcommand{\bara}[1]{\overline{#1}}

\newcommand{\Sf}[1]{\mathsf{#1}}
\newcommand{\bT}{\boldsymbol{T}}

\newcommand{\bi}{\boldsymbol{i}}
\newcommand{\bt}{\boldsymbol{t}}

\begin{document}
\allowdisplaybreaks

\title{Morita base change in Hopf-cyclic (co)homology} 

\author{Laiachi El Kaoutit}
\author{Niels Kowalzig}

\address{L.E.K.: 
Universidad de Granada, Departamento de \'Algebra, Facultad de Educaci\'on y Humanidades de Ceuta,
El Greco ${\rm N}^{\rm \underline{o}}$~10, 51002 Ceuta, Espa\~na}

\email{kaoutit@ugr.es}

\address{N.K.: 
Universidad de Granada, Departamento de \'Algebra, Facultad de Ciencias, 18071 Granada, Espa\~na}

\email{kowalzig@ihes.fr}


\begin{abstract}
In this paper, we establish the invariance of cyclic (co)homology of left Hopf algebroids under the change of Morita equivalent base algebras. 
The classical 
result on Morita invariance for cyclic homology of associative algebras 
appears as a special example of this theory. 
 In our main application we consider the Morita equivalence between the algebra of complex-valued smooth functions on the classical $2$-torus and the  coordinate algebra of the noncommutative $2$-torus {with rational parameter}. We then construct a Morita base change left Hopf algebroid over {this} noncommutative $2$-torus and  show that 
its cyclic (co)homology  can be computed by means of the homology of the Lie 
algebroid of vector fields on the classical $2$-torus.
\end{abstract}

\keywords{Morita equivalence; cyclic homology; Hopf algebroids; vector bundles;  Lie algebroids; noncommutative tori. \\  {\footnotesize Research of L.~El Kaoutit was supported by the grant MTM2010-20940-C02-01 from the Ministerio de Ciencia e Innovaci\'on and from FEDER. 
N.~Kowalzig acknowledges funding by the Excellence Network of the University of Granada (GENIL)}}
\maketitle

\section{Introduction}

The concept of left Hopf algebroids provides a natural
framework for unifying and extending classical constructions in homological
algebra. Group, groupoid, Lie algebra, Lie algebroid and Poisson (co)homology, Hochschild and cyclic homology for associative algebras, as well as Hopf-cyclic homology for Hopf algebras, are all special cases of the cyclic homology of left Hopf algebroids 
since the rings over which these theories can be expressed as derived functors are all left Hopf algebroids 
(see, for example, \cite{BoeSte:CCOBAAVC, ConMos:HACCATTIT, Cra:CCOHA, Kow:HAATCT, KowKra:CSIACT, KowPos:TCTOHA} for more details).

As for every (co)homology theory it is an interesting issue to examine its behaviour under (any suitable notion of) Morita equivalence.  Nevertheless, a satisfactory notion of  Morita equivalence between two possibly noncommutative left Hopf algebroids is up to our knowledge  far from being obvious. The difficulty comes out when, for instance, one tries to understand how the notion of Morita equivalence between two Lie algebroids, in the sense of  \cite{Cra:DAACVEIACC, Gin:GGOPVB} and others, can be reflected  to their respective associated (universal) left Hopf algebroids in such a way that invariant properties, especially homological ones, between equivalent Lie algebroids remain invariant at the level of left Hopf algebroids. 
In the commutative case, that is, for  commutative Hopf algebroids, several notions already  exist in the literature, see, e.g., \cite{Hovey:02, Hovey/Strickland:05}.

In this paper, we restrict ourselves to the case of Morita base change left Hopf algebroids. That is, we  study from a cyclic (co)homology point of view two Morita equivalent left Hopf algebroids of the form $(R,U) \sim (S,\tilde{U})$, where $R  \sim S$ are Morita equivalent base rings and $\tilde{U}$ is constructed from $U$. It is worth noticing that for the  case  of commutative Hopf algebroids or Hopf algebras,  this notion reduces to simply changing the base ring by an isomorphism. 
Nevertheless, this restriction is not far from geometric applications since, for example, the algebra of smooth functions on a smooth manifold $\cM$ is Morita equivalent to the endomorphism algebra of global smooth sections of a vector bundle on $\cM$. 
More precisely, one can start with a smooth vector bundle $\cP \to \cM$ and a Lie algebroid $(\cM, \cE)$, then associate to them a Morita base change $(\cC^{\infty}(\cM),\cV\Gamma(\cE)) \sim ({\rm End}(\Gamma(\cP)), \widetilde{\cV\Gamma(\cE))})$, where   $(\cC^{\infty}(\cM)
 ,\cV\Gamma(\cE))$ is the associated (universal) left Hopf algebroid attached to $(\cM, \cE)$, see Section 5. In the aim of illustrating our methods, we give an explicit application concerning the noncommutative $2$-torus {with rational parameter}. 

A left Hopf algebroid ($\times_R$-Hopf algebra) $U$ is, roughly speaking, a Hopf algebra whose ground ring is not
a commutative ring $\kk$ but a possibly noncommutative $\kk$-algebra $R$, see \cite{Boe:HA,Schau:DADOQGHA, Tak:GOAOAA}. In categorical terms, $U$ is a ring extension of the enveloping ring $\Reh=R\tensor{\Bbbk}R^{o}$ of the base algebra $R$, where the category of left $U$-modules is a right closed monoidal category, and the forgetful functor to the category of left $\Reh$-modules is strict monoidal and preserves right inner hom-functors.
As $\kk$-bialgebras are underlying Hopf algebras, (left) $R$-bialgebroids are the underlying structure of (left) Hopf algebroids, but for bialgebroids the forgetful functor is in general not right inner-hom preserving. 

Morita base change for bialgebroids (following \cite{Schau:MBCIQG}) provides a possibility to produce new bialgebroids by replacing the base algebra $R$ by a Morita equivalent base algebra $S$ in such a way that the resulting $R$-bialgebroid has a monoidal category of representations equivalent to that of the original $R$-bialgebroid. More generally, the base algebra $R$ can be replaced by a $\sqrt{\mbox{Morita}}$ equivalent algebra $S$, see \cite{Tak:squareMorita}: two algebras are $\sqrt{\mbox{Morita}}$ equivalent if one has an equivalence of $\kk$-linear monoidal categories of bimodules $ {}_{R^{\Sf{e}}}{\rm Mod}  \simeq {}_{S^{\Sf{e}}}{\rm Mod}$.
Such an equivalence relation between two bialgebroids is weaker than to consider two bialgebroids to be equivalent if their monoidal categories of (co)representations are so. In particular, Morita base change establishes a relation between two bialgebroids
  in a way that is meaningless for ordinary $\kk$-bialgebras, as already
said before.

Apart from what we mentioned above, the importance of the notion of Morita base change moreover consists in unifying seemingly different concepts: for example, every weak $\C$-bialgebra (which can be considered as bialgebroids \cite[\S3.2.2]{Boe:HA}) can be shown to be a face algebra (which are examples of 
bialgebroids as well \cite{Schau:FAATRB}) up to Morita base change \cite[\S5.2]{Schau:MBCIQG}. Here we present no application in this direction, this will be left for a future project.

Useful for our purposes is the fact that Morita base change equivalence carries over to the 
Hopf structure: an $R$-bialgebroid is left Hopf if and only if its Morita base change equivalent $S$-bialgebroid is left Hopf as well \cite[Prop.~4.6]{Schau:MBCIQG}.

In this paper, we will consider  the cyclic (co)homology for left Hopf algebroids from \cite{KowKra:CSIACT} and confront it with the Morita base change theory from \cite{Schau:MBCIQG}. Our aim is to give, in the spirit of \cite{McC:MEACH}, the explicit chain morphisms and chain homotopies that establish equivalences of (co)cyclic modules between the original left Hopf algebroid and the Morita base change left Hopf algebroid $\td{U}$; see, however, Remark \ref{cat} for a comment on a categorical approach. 
As a consequence, we obtain our central theorem which we copy here, see the main text for the details and in particular the notation used:

{\renewcommand{\thetheorem}{\bf{A}}

\begin{thm}{\em (Morita base change invariance of (Hopf-)cyclic (co)homology)}
Let $(R, U)$ be a left Hopf algebroid, $M$ a left $U$-comodule right $U$-module which is stable anti Yetter-Drinfel'd, and $(R,S,P,Q,\phi,\psi)$ a Morita context. Consider  its induced  $\sqrt{\mbox{Morita}}$ context $(R^{\Sf{e}},\Se,\Pe,\Qe,\phi^{\Sf{e}},\psi^{\Sf{e}})$ and the Morita base change left Hopf algebroid $(S,\tilde{U}:= \Pe \tensor{\Ree}U\tensor{\Ree}\Qe)$. Then
$$
\begin{array}{rclcrcl}
H_\bull(U, M) &\simeq& H_\bull(\td{U}, P \otimes_\erre M \otimes_\erre Q), && H^\bull(U, M) &\simeq& H^\bull(\td{U}, P \otimes_\erre M \otimes_\erre Q), \\
HC_\bull(U, M) &\simeq& HC_\bull(\td{U}, P \otimes_\erre M \otimes_\erre Q), && HC^\bull(U, M) &\simeq& HC^\bull(\td{U}, P \otimes_\erre M \otimes_\erre Q)
\end{array}
$$
are isomorphisms of $\Bbbk$-modules.
\end{thm}
}
As an application, we first indicate how the classical result of Morita invariance for cyclic homology of associative algebras (see, e.g., \cite{Con:NCDG, DenIgu:HHATSOFP, McC:MEACH}) fits into our general theory. 

Second, we consider a Morita context between the complex-valued smooth functions on the commutative real $2$-torus  $\mathbb{T}^2$ 
and the coordinate ring of the noncommutative $2$-torus with rational parameter. After reviewing the construction for this case, we apply the Morita invariance to the universal left Hopf algebroid associated to the  Lie algebroid of vector fields on $\mathbb{T}^2$ and its Morita base change left Hopf algebroid $\widetilde{\cV K}$ over this noncommutative $2$-torus, which establishes a passage from commutative to noncommutative geometry: in the spirit of considering left Hopf algebroids as the noncommutative analogue of Lie algebroids and their primitive elements as the noncommutative analogue of (generalised) vector fields, the primitive elements of $\widetilde{\cV K}$ can be seen to consist of vector fields on the noncommutative torus.
{\renewcommand{\thetheorem}{\bf{B}}
\begin{cor}
Let $\Sf{q} \in \mathbb{S}^1$ be a root of unity,
and consider the Lie algebroid $\big(R=C^\infty(\mathbb{T}^2),K = \Der_\C(C^\infty(\mathbb{T}^2))\big)$ of vector fields on the complex torus $\mathbb{T}^2$ and its associated left Hopf algebroid $(R,\cV K)$. 
Let $M$ be a right $\cV K$-module 
and $(R,S,P,Q,\phi,\psi)$ the Morita context of Eq.~\eqref{Eq: Morita torus}. 
We then have the following natural $\C$-module isomorphisms
$$
\begin{array}{rclcrcl}
H_\bull(\cV K , M) &\simeq & H_\bull( \widetilde{\cV K}, \td{M}), && HC_\bull(\cV K, M) &\simeq& HC_\bull( \widetilde{\cV K}, \td{M}), \\
H^\bull(\cV K , M) &\simeq & H^\bull( \widetilde{\cV K}, \td{M}), && HC^\bull(\cV K, M) &\simeq& HC^\bull( \widetilde{\cV K}, \td{M}), 
\end{array}
$$
where $\widetilde{\cV K}$ is the Morita base change left Hopf algebroid over the noncommutative torus $\cC^{\infty}(\mathbb{T}_{\Sf{q}}^2)$ whose structure maps are given as in \S\ref{subsect:Morita}.\\
Furthermore,  assume that $M$ be $R$-flat. Then we have that
$$
\begin{array}{rclcrcl}
H_\bull( \widetilde{\cV K}, \td{M}) &\!\!\!\!\simeq&\!\!\!\!  H_\bull(K, M), && 
HC_\bull( \widetilde{\cV K}, \td{M}) &\!\!\!\!\simeq&\!\!\!\! \textstyle\bigoplus_{i \geq 0}H_{\bull -2i}(K,M), \\
H^\bull( \widetilde{\cV K}, \td{M}) &\!\!\!\!\simeq&\!\!\!\!  M \otimes_\erre \textstyle\bigwedge^\bull_\erre K, && 
HP^\bull( \widetilde{\cV K}, \td{M}) &\!\!\!\!\simeq&\!\!\!\! \textstyle\bigoplus_{i \equiv \bull {\rm mod} 2}H_{i}(K,M) \\
\end{array}
$$ 
are natural $\mathbb{C}$-module isomorphisms,
where $H_\bull(K, M) :=  \Tor^{\cV K}_\bull(M, R)$,
 and where $HP^\bull$ denotes periodic cyclic cohomology.
\end{cor}
}

\smallskip

\textbf{Acknowledgements.}  
It is a pleasure to thank J.~G\'omez-Torrecillas, M.~Khalkhali, U.~Kr\"ahmer, and A.~Weinstein for stimulating discussions and comments.  The authors are also grateful to the referee for the careful reading of the manuscript and the useful comments.

\section{Preliminaries}
\subsection{Some conventions}
Throughout this note, ``ring'' means associative algebra over a fixed commutative ground ring $\Bbbk$. 
All other algebras, modules etc., will have an underlying structure of a central 
$\Bbbk$-module. 
Given a ring $R$, we denote by ${}_R\Mod$ the category of 
left $R$-modules, 
by $R^{o}$ the 
opposite ring and by
$R^\mathrm{e} := R \otimes_{\Bbbk} R^{o}$ 
the enveloping algebra 
of $R$. 
An {\em $R$-ring} is a monoid
in the monoidal
category $({}_{\Reh}\Mod, \otimes_\erre, R)$ of $\Reh$-modules (i.e., $(R,R)$-bimodules
with symmetric action of $\Bbbk$), fulfilling
associativity and unitality. Likewise,  
an {\em $R$-coring} is a comonoid in
$({}_{\Reh}\Mod, \otimes_\erre, R)$,  
fulfilling coassociativity and
counitality. 

Our main object is an $\Reh$-ring $U$.
Explicitly, such an $\Reh$-ring is given by a $\Bbbk$-algebra 
homomorphism 
$
		  \eta=\eta_U : R^\mathrm{e} \rightarrow U
$ 
whose restrictions
\begin{equation}
\label{wind&rain}
		  \sll:= \eta( - \otimes_{\Bbbk} 1) : 
		  R \to U 
		  \quad \mbox{and} \quad 
		  \tl := \eta(1 \otimes_{\Bbbk} -) : 
		  \Rop \to U
\end{equation}
will be called the {\em source} and {\em
target} map, respectively. Left and right multiplication in $U$
give rise to an 
$(\Reh,\Reh)$-bimodule structure on $U$,
that is, four  actions of $R$
that we denote by 
\begin{equation*}\label{bimod-lmod}
		  r \lact u \ract r' :=\sll(r)t(r')u,\quad 
		  r \blact u \bract r'
		  :=u\sll(r')t(r), 
		  \quad r,r'\in R, \ u \in U,
\end{equation*}
which are commuting, in the sense that, for every $a, a', r, r' \in R$ and $u,v \in U$, we have 
\begin{eqnarray}
\label{pouvoir}
a' \blact\lr{ r \lact u \ract r'} \bract a &=& r \lact \lr{a'\blact u \bract a} \ract r'; \\
\nonumber
\lr{u \bract r} \lr{ v \ract a} &=& \lr{a \blact u} \lr{r \lact v}. 
\end{eqnarray}

If not stated otherwise, we view $U$ as
an $(R,R)$-bimodule using the actions $\lact,\ract$, denoted ${}_{\lact}U_{\ract}$.
In particular, we define 
the tensor product
$U \otimes_\erre U$
with respect to this bimodule structure.
On the other hand, using the
actions $\blact, \bract$ permits to define the 
{\em Sweedler-Takeuchi product}, see \cite{Swe:GOSA, Tak:GOAOAA}:
\begin{equation*}
\label{taki}
		  U \times_\erre U :=  
		  \left\{\underset{}{} \textstyle\sum_i u_i \otimes_\erre
		  v_i 
		  \in U \otimes_\erre U \mid 
		  \sum_i r \blact u_i 
		  \otimes_\erre v_i = 
		  \sum_i u_i \otimes_\erre v_i \bract r, 
		  \ \forall r \in R\right\}.
\end{equation*}

One easily verifies that $U \times_\erre U$ is an $\Reh$-ring via factorwise multiplication, with unit element $1_\uhhu \otimes_\erre 1_\uhhu$ and $\eta_{\uhhu \times_{\scriptscriptstyle{\erre}} \uhhu}(r \otimes_\Bbbk r') = s(r) \otimes_\erre t(r')$, for $r, r' \in R$.

\subsection{Bialgebroids} \cite{Tak:GOAOAA}
Bialgebroids are a generalisation of
bialgebras. An important subtlety 
is that the algebra and coalgebra 
structure are defined in
different monoidal categories. 

\begin{definition}\label{left-bialg}
Let $R$ be a $\Bbbk$-algebra.
A {\em left bialgebroid} over $R$ 
is an $\Reh$-ring $U$
together with two homomorphisms of
 $\Reh$-rings
$$
		  \gD : U \rightarrow U
		  \times_\erre U,\quad
		  \hat \varepsilon : U \rightarrow \mathrm{End}_{\Bbbk}(R)  
$$
which turn $U$ into an $R$-coring 
with coproduct $\gD$ (viewed as a
 map $U \rightarrow  U_\ract  \otimes_\erre \due U \lact {}$) and
 counit $\varepsilon : U \rightarrow R$,
$u \mapsto (\hat\varepsilon (u))(1)$.
\end{definition}

So one has for example for $u \in U$, $r, r' \in R$
\begin{equation}
\label{panuelos}
		  \Delta (r \lact u \ract r')=
		  r \lact u_{(1)} \otimes_{\erre} 
		  u_{(2)} \ract r',\quad
		  \Delta (r \blact u \bract r')=
		  u_{(1)} \bract r' \otimes_{\erre} r \blact u_{(2)},
\end{equation}
using Sweedler's shorthand 
notation 
$u_{(1)} \otimes_{\erre} u_{(2)}$ 
for $\Delta(u)$, as well as in $U \times_{\erre} U$ the identity
\begin{equation}
\label{takinew}
		  r \blact u_{(1)} \otimes_{\erre} u_{(2)} =
		  u_{(1)} \otimes_{\erre} u_{(2)} \bract r.
\end{equation}
The counit, on the other hand, fulfills for any $u, v \in U$ and $r, r' \in R$
\begin{equation}\label{Eq:counit}
\varepsilon(r \lact u \ract r') = r\varepsilon(u)r', \quad \varepsilon (u \bract r) =  \varepsilon(r \blact u), \quad
\varepsilon(uv) = \varepsilon(u \bract \varepsilon (v)) =   \varepsilon(\varepsilon (v) \blact u).
\end{equation}

\subsection{Left Hopf algebroids} {\cite{Schau:DADOQGHA}}
Left Hopf algebroids have been
introduced by Schauenburg 
under the name {\em $\times_\erre$-Hopf
algebras} and 
generalise Hopf 
algebras towards left bialgebroids.
For a left bialgebroid $U$ over $R$, one
defines the
{\em (Hopf-)Galois map} 
\begin{equation*}
\label{Galois}
\beta: {}_\blact U \otimes_\Ropp U_\ract \to 
U_\ract \otimes_\erre {}_\lact U, \quad u
\otimes_\Ropp v \mapsto  
u_{(1)}  \otimes_\erre u_{(2)}v, 
\end{equation*}
where 
\begin{equation}
\label{tata}
 {}_\blact U \otimes_\Ropp U_\ract = U
 \otimes_{\Bbbk} U/
{{\rm
 span}\{r \blact u \otimes_{\Bbbk} v - u
 \otimes_{\Bbbk} v \ract r \mid
u,v \in U, r \in R \}}.
\end{equation}

\begin{dfn}\cite{Schau:DADOQGHA}
\label{hopftimesleft}
A left $R$-bialgebroid $U$ is 
called a 
{\em left Hopf algebroid} 
(or {\em $\times_\erre$-Hopf algebra}) if
$\beta$ is a bijection. 
\end{dfn}

\noindent By means of a Sweedler-type notation 
\begin{equation*}
\label{pm}
		  u_+ \otimes_\Ropp u_- := 
		  \beta^{-1}( u \otimes_\erre 1)
\end{equation*}
for the translation map
$
		  \beta^{-1}(- \otimes_\erre 1) : U \rightarrow 
		  {}_\blact U \otimes_\Ropp U_\ract,
$ 
one obtains for all $u, v \in U$, $r, r'\in R$ the following useful identities
\cite[Prop.~3.7]{Schau:DADOQGHA}:
\begin{eqnarray}
\label{Sch1}
u_{+(1)} \otimes_\erre u_{+(2)} u_- &=& u \otimes_\erre 1 \in U_\ract \otimes_\erre {}_\lact U, \\
\label{Sch2}
u_{(1)+} \otimes_\Ropp u_{(1)-} u_{(2)}  &=& u \otimes_\Ropp  1 \in  {}_\blact U
\otimes_\Ropp  U_\ract, \\ 
\label{Sch3}
u_+ \otimes_\Ropp  u_- & \in 
& U \times_\Rop U, \\
\label{Sch38}
u_{+(1)} \otimes_\erre u_{+(2)} \otimes_\Ropp  u_{-} &=& u_{(1)} \otimes_\erre
u_{(2)+} \otimes_\Ropp  u_{(2)-},\\
\label{Sch37}
u_+ \otimes_\Ropp  u_{-(1)} \otimes_\erre u_{-(2)} &=& u_{++} \otimes_\Ropp
u_- \otimes_\erre u_{+-}, \\
\label{Sch4}
(uv)_+ \otimes_\Ropp  (uv)_- &=& u_+v_+
\otimes_\Ropp  v_-u_-, 
\\ 
\label{Sch47}
u_+u_- &=& \sll (\varepsilon (u)), \\
\label{Sch48}
u_+ \tl (\varepsilon (u_-)) &=& u, \\
\label{Sch5}
(\sll (r) \tl (r'))_+ \otimes_\Ropp  (\sll (r) \tl (r') )_- 
&=& \sll (r) \otimes_\Ropp  \sll (r'), 
\end{eqnarray}
where in (\ref{Sch3}) we mean the Sweedler-Takeuchi product
\begin{equation*}
\label{petrarca}
		  U \times_\Rop U:=
		  \left\{\underset{}{} \textstyle\sum_i u_i \otimes_\Ropp  v_i \in 
		  {}_\blact U \otimes_\Ropp  U_\ract\,|\,
		  \sum_i u_i \ract r \otimes_\Ropp  v_i=
		  \sum_i u_i \otimes_\Ropp  r \blact
		  v_i, \,\, \forall r \in R
		  \right \},
\end{equation*}
which is an algebra by factorwise
multiplication, but with opposite 
multiplication on the second factor.
Note that in (\ref{Sch37}) the tensor product
over $R^{o}$ links the first and
third tensor component. 
By (\ref{Sch1}) and (\ref{Sch3}), one can write
$$
\beta^{-1}(u \otimes_\erre v) = u_+ \otimes_\Ropp  u_-v.
$$

\subsection{$U$-modules}
Let $(R, U)$ be a left bialgebroid. 
Left and right $U$-modules are defined as modules over the ring $U$, with respective actions denoted by juxtaposition. 
We denote the respective categories by
${}_U\Mod$ and ${}_{U^\op}\Mod$; while ${}_U\Mod$ is
a monoidal category, ${}_{U^\op}\Mod$ in general is
not \cite{Schau:BONCRAASTFHB}. One has a forgetful functor  
${}_{U}\Mod \rightarrow {}_{\Reh}\Mod$ using which
we consider every left 
$U$-module $N$ also as an $(R,R)$-bimodule
with actions 
\begin{equation}\label{brot}
	anb := 	  a \lact n \ract b := \sll(a)\tl(b)n,\quad
		  a,b \in R,n \in N.
\end{equation} 
Similarly, every right $U$-module $M$ is also 
an $(R,R)$-bimodule via
\begin{equation}
\label{salz}
	amb:= 	  a \blact m \bract b := 
			m \sll(b) \tl(a),\quad
		  a,b \in R,m \in M,
\end{equation} 
and in both cases we usually prefer to express these actions just by juxtaposition if no ambiguity is to be expected.

\subsection{$U$-comodules}

Similarly as for coalgebras, one may
define comodules over 
bialgebroids, but the underlying
$R$-module structures need some extra
attention. For the following definition
confer e.g.\
\cite{Schau:BONCRAASTFHB,Boe:GTFHA,
BrzWis:CAC}.

\begin{dfn}
\label{tempo}
A {\em left $U$-comodule} for a left bialgebroid $(R, U)$ 
is a left comodule of the underlying $R$-coring $(U, \gD,
\varepsilon)$, 
i.e., a left $R$-module $M$ with action $L_\erre: (r,m) \mapsto rm$ and a
left $R$-module map  
\begin{equation*}
		  \gD_\emme: 
		  M \to U_\ract \otimes_\erre M, 
		  \quad 
		  m \mapsto m_{(-1)} \otimes_\erre m_{(0)}
\end{equation*}
satisfying the usual coassociativity and
 counitality axioms. We denote the category of left
 $U$-comodules 
by ${}_U\bf Comod$.
\end{dfn}

\noindent On any left $U$-comodule
one can additionally define 
a right $R$-action 
\begin{equation}
\label{grimm}
mr := \varepsilon\big(m_{(-1)}\bract r\big)m_{(0)}.
\end{equation}
This action originates in fact from the algebra morphism
\begin{equation}
\label{Eq:U*}
 R \to U^*, \quad
 r \mapsto \left[u \mapsto \varepsilon(u\bract r) \right],
\end{equation}
where $U^*:=\Hom_{R}(U_\ract,R_\erre)$ is the right convolution ring of the underlying $R$-coring $U$, and the canonical functor ${}_U{\bf Comod} \to \Mod_{U^*}$ that endows any left $U$-comodule $X$ with a right $U^*$-action  given  by 
$$
x\, \sigma \,\, =\,\, \sum_{(x)} \sigma(x_{(-1)}) x_{(0)}
$$
for every $x \in X$ and $\sigma \in U^*$. 
The above action is then the restriction to scalars associated to the algebra morphism \eqref{Eq:U*}, 
and the action \eqref{grimm} is the unique one that turns $M$ into a left
$\Reh$-module in such a way that the coaction is an
$\Reh$-module morphism
$$
		  \gD_\emme: M \to U \times_\erre M, 
$$
where $U \times_\erre M$ is the Sweedler-Takeuchi product
$$
U \times_\erre M := \left\{\underset{}{} \textstyle\sum_i u_i \otimes_\erre m_i		 
		  \in U \otimes_\erre M \mid
		  \sum_i u_i\tl (a) \otimes_\erre m_i
		  = \sum_i u_i \otimes_\erre m_i a, \
		  \forall a \in R\right\}.
$$
In other words, $M$ becomes  a left  $\times_R$-$U$-comodule. Conversely, any left $\times_R$-$U$-comodule gives rise to a left $U$-comodule. This correspondence establishes  in fact an isomorphism of categories.

As a result of the previous discussion, $\Delta_\emme$ satisfies the identities 
\begin{eqnarray}
\label{maotsetung}
		  \gD_\emme(rmr') &=& 
		  \lr{r \lact  m_{(-1)} \bract r'} \otimes_\erre
		  m_{(0)}, \\
		  \label{douceuretresistance}
		  m_{(-1)} \otimes_\erre m_{(0)}r
		  &=&  
		 \lr{ r \blact m_{(-1)}}  \otimes_\erre m_{(0)}.
\end{eqnarray}

\subsection{Cyclic homology for left Hopf algebroids}

\subsubsection{Stable anti Yetter-Drinfel'd modules}

The following definition is the left
bialgebroid right module and left
comodule version of the corresponding notion in
\cite{BoeSte:CCOBAAVC, HajKhaRanSom:SAYD}.

\begin{dfn}
\label{SAYD}
Let $(R, U)$ be a left Hopf algebroid, and let
 $M$ simultaneously be a left
 $U$-comodule and a right $U$-module with
 action denoted by $(m, u) \mapsto mu$
 for $u \in U$, $m \in M$. We call $M$
 an {\em anti Yetter-Drinfel'd (aYD)
 module} if: 
\begin{enumerate}
\item
The two $\Reh$-module structures on $M$ originating from its nature as $U$-comodule resp.\ right $U$-module coincide:
for all $r, r' \in R$, $m \in M$
\begin{eqnarray}
\label{campanilla1}
rm &=& r \blact m, \\
\label{campanilla2}
mr' &=& m \bract r',
\end{eqnarray}
where the right $R$-module structure on the left hand side is given by \rmref{grimm}. 
\item
For $u \in U$ and $m \in M$ one has the following compatibility between action and coaction:
\begin{equation}
\label{huhomezone}
		  \gD_\emme(mu) = 
		  u_- m_{(-1)} u_{+(1)} \otimes_\erre m_{(0)} u_{+(2)}.
\end{equation}
\end{enumerate}
The aYD module $M$ is
 said to be {\em stable (SaYD)} if, 
for all $m \in M$, one has
\begin{equation}
\label{Eq:SaYD}
m_{(0)}m_{(-1)} = m.
\end{equation}
\end{dfn}

\subsubsection{Cyclic (co)homology}
\label{almeria}
We will not recall the formalism of cyclic (co)homology in
full detail; see, e.g., \cite{FeiTsy:AKT, Lod:CH} for more information. However,
recall that 
para-(co)cyclic $\kk$-modules generalise 
(co)cyclic $\kk$-modules by dropping the
condition that the (co)cyclic operator
implements an action of
$\mathbb{Z}/(n+1)\mathbb{Z}$ on the
degree $n$ part. Thus a para-cyclic
$\kk$-module is a simplicial
$\kk$-module
$(C_\bull,d_\bull,s_\bull)$ 
and a para-cocyclic $\kk$-module is a
cosimplicial $\kk$-module
$(C^\bull,\delta_\bull,\sigma_\bull)$,
together with $\kk$-linear maps 
$t_n : C_n \rightarrow C_n$ resp.~$\tau_n : C^n \rightarrow C^n$
satisfying, respectively
\begin{equation}\label{belleville}
\!\!\!\!\!\!\!\!
\begin{array}{cc}
\begin{array}{rcl}
d_i \circ t_n  &\!\!\!\!\!\!=&\!\!\!\!\!\! \left\{\!\!\!
\begin{array}{ll}
t_{n-1} \circ d_{i-1} 
& \!\!\!\! \mbox{if} \ 1 \leq i \leq n, \\
 d_n & \!\!\!\! \mbox{if} \
i = 0,
\end{array}\right. \\
\\
s_i \circ t_n &\!\!\!\!\!\!=&\!\!\!\!\!\! \left\{\!\!\!
\begin{array}{ll}
t_{n+1} \circ s_{i-1} & \!\!\!\!
\mbox{if} \ 1 \leq i \leq n, \\
 t^2_{n+1} \circ s_n
 & \!\!\!\! \mbox{if} \
i = 0, \\
\end{array}\right.
\end{array}
\!\!\!\!&\!\!\!\!
\begin{array}{rcll}
\tau_n \circ \gd_i &\!\!\!\!\!=&\!\!\!\!\! \left\{\!\!\!
\begin{array}{l}
\gd_{i-1}\circ \tau_{n-1} \\
 \gd_n 
\end{array}\right. & \!\!\!\!\!\!\!\!\! \begin{array}{l} \mbox{if} \ 1
\leq i \leq n, 
 \\ \mbox{if} \ i = 0,
 \end{array} \\
\\
\tau_n \circ \sigma_i &\!\!\!\!\!=&\!\!\!\!\! \left\{\!\!\!
\begin{array}{l}
\sigma_{i-1} \circ \tau_{n+1} \\
 \sigma_n \circ \tau^2_{n+1} 
\end{array}\right. & \!\!\!\!\!\!\!\!\!
\begin{array}{l} \mbox{if} \ 1 \leq i
 \leq n,  
\\ \mbox{if} \ i = 0. \end{array} 
\end{array}
\end{array}
\end{equation}
Such a para-(co)cyclic module is called {\em (co)cyclic} if $t^{n+1}_n = \id$ (resp.\ $\tau^{n+1}_n = \id$).
Any cyclic module $C_\bull$ 
gives rise to a cyclic bicomplex $C_{\bull\bull}$, see, e.g., \cite{FeiTsy:AKT} for details. The only thing we recall here is that the differential on the $b$-columns is given by
\begin{equation}
\label{Eq:b}
b = \sum_{i=0}^{n} (-1)^i\, d_i, 
\end{equation}
and likewise $\beta := \sum^{n+1}_{i=0} (-1)^i \gd_i$ for a cocyclic module.

\subsubsection{The para-(co)cyclic module associated to a left Hopf algebroid 
{\rm (\cite{KowKra:CSIACT}, cf.~also \cite{KowPos:TCTOHA})}}\label{ssubsect:CUM}

Let $M$ be simultaneously a left $U$-comodule and a right $U$-module with compatible left $R$-action as in \rmref{campanilla1}.
Set
$$
C_\bull(U,M) \, := \,\,M \otimes_\Ropp  (\due U \blact \ract)^{\otimes_\Ropp  \bull},
$$
and in each degree $n$ define the following structure maps on it:
\!\!\!\!
\begin{equation}
\label{adualnightinpyongyang}
\!\!\!
\begin{array}{rcll}
d_i(m \otimes_\Ropp x)  &\!\!\!\! =&\!\!\!\! 
\left\{ \!\!\!
\begin{array}{l}
m \otimes_\Ropp u^1  \otimes_\Ropp   \cdots   \otimes_\Ropp   \big(\varepsilon(u^n) \blact u^{n-1}\big)
\\
m \otimes_\Ropp \cdots \otimes_\Ropp  (u^{n-i} u^{n-i+1})
 \otimes_\Ropp  \cdots \otimes_\Ropp u^n
\\
(mu^1) \otimes_\Ropp u^2  \otimes_\Ropp   \cdots    \otimes_\Ropp  
u^n 
\end{array}\right.  & \!\!\!\!\!\!\!\!\!\! \,  \begin{array}{l} \mbox{if} \ i \!=\! 0, \\ \mbox{if} \ 1
\!  \leq \! i \!\leq\! n-1, \\ \mbox{if} \ i \! = \! n, \end{array} \\
\ \\
s_i(m \otimes_\Ropp x) &\!\!\!\! =&\!\!\!\!  \left\{ \!\!\!
\begin{array}{l} m  \otimes_\Ropp   u^1  
\otimes_\Ropp   \cdots   \otimes_\Ropp
 u^n  \otimes_\Ropp 1
\\
m \otimes_\Ropp \cdots \otimes_\Ropp   u^{n-i} 
\otimes_\Ropp   1  
\otimes_\Ropp   u^{n-i+1}  \otimes_\Ropp
 \cdots  \otimes_\Ropp u^n
\\
m \otimes_\Ropp 1 
\otimes_\Ropp u^1  \otimes_\Ropp   \cdots    \otimes_\Ropp  u^n 
\end{array}\right.   & \!\!\!\!\!\!\!\!\!  \begin{array}{l} 
\mbox{if} \ i\!=\!0, \\ 
\mbox{if} \ 1 \!\leq\! i \!\leq\! n-1, \\  \mbox{if} \ i\! = \!n, \end{array} \\
\ \\
t_n(m \otimes_\Ropp x) 
&\!\!\!\!=&\!\!\!\!
(m_{(0)} u^1_+) \otimes_\Ropp u^2_+ \otimes_\Ropp  \cdots  \otimes_\Ropp u^n_+ \otimes_\Ropp (u^n_- \cdots u^1_- m_{(-1)}),  
& \\
\end{array}
\end{equation}
where we abbreviate $x:=u^1 \otimes_\Ropp \cdots \otimes_\Ropp u^n$. 
As explained in detail in \cite{KowKra:CSIACT}, this cyclic module is the generalised ``cyclic dual'' to the following cocyclic module:
set
$$
C^\bull(U,M) \, := \,\, (\due U \lact \ract)^{\otimes_\erre  \bull} \otimes_\erre M,
$$
with structure maps in degree $n$ given by
\begin{equation}
\!\! \begin{array}{rll}
\label{anightinpyongyang}
\gd_i(z \otimes_\erre m) \!\!\!\!&= \left\{\!\!\!
\begin{array}{l} 1
\otimes_\erre u^1 \otimes_\erre \cdots
 \otimes_\erre u^n \otimes_\erre m  
\\ 
u^1 \otimes_\erre \cdots \otimes_\erre \Deltaell (u^i) \otimes_\erre \cdots
 \otimes_\erre u^n \otimes_\erre m
\\
u^1 \otimes_\erre \cdots \otimes_\erre u^n \otimes_\erre m_{(-1)} \otimes_\erre m_{(0)} 
\end{array}\right. \!\!\!\!\!\!\!\! 
& \!\! \hspace*{-2cm}  \begin{array}{l} \mbox{if} \ i=0, \\ \mbox{if} \
  1 \leq i \leq n, \\ \mbox{if} \ i = n + 1,  \end{array} \\
\\
\gd_j(m) \!\!\!\! &= \left\{ \!\!\!
\begin{array}{l}
		  1
		  \otimes_\erre m  \quad
\\
m_{(-1)} \otimes_\erre m_{(0)}  \quad 
\end{array}\right. & \!\! \hspace*{-2cm} 
\begin{array}{l} \mbox{if} \ j=0, \\ \mbox{if} \
  j = 1 ,  \end{array} \\
\\
\gs_i(z \otimes_\erre m) \!\!\!\! 
&= u^1 \otimes_\erre \cdots \otimes_\erre
{\varepsilon} (u^{i+1}) \otimes_\erre \cdots \otimes_\erre u^n \otimes_\erre m & \! \,
\hspace*{1pt} \hspace*{-2cm}  0 \leq i \leq n-1,
\\ 
\\
\tau_n(z \otimes_\erre m) \!\!\!\! 
&= u^1_{-(1)}u^2 \otimes_\erre \cdots \otimes_\erre u^1_{-(n-1)}u^n\otimes_\erre u^1_{-(n)}m_{(-1)} \otimes_\erre m_{(0)}u^1_+, & 
\end{array}
\end{equation}
where we abbreviate $z:=u^1 \otimes_\erre
\cdots \otimes_\erre u^n$. 

In \cite{KowKra:CSIACT} it was shown that, under the minimal assumption \rmref{campanilla1}, the maps \rmref{adualnightinpyongyang} (resp.~\rmref{anightinpyongyang}) 
give rise to a para-cyclic (resp.~para-cocylic) module, which is cyclic (resp.~cocyclic) if $M$ is SaYD, i.e., additionally fulfills \rmref{campanilla2}--\rmref{Eq:SaYD}.

Let us denote by $H_\bull(U,M)$ and $HC_\bull(U,M)$ the resulting simplicial and cyclic homology groups of $C_\bull(U,M)$, and likewise by  $H^\bull(U,M)$ and $HC^\bull(U,M)$ the resulting simplicial and cyclic cohomology groups of $C^\bull(U,M)$.

\section{$\sqrt{\mbox{Morita}}$ theory and Morita base change Hopf algebroids}
In this section, we first recall some  general facts about Morita contexts and their induced $\sqrt{\mbox{Morita}}$ theory   in the sense of Takeuchi \cite{Tak:squareMorita}.  Secondly, we explain how this theory 
was used by Schauenburg to introduce Morita base change (left) Hopf algebroids in \cite{Schau:MBCIQG}. 
In order to establish our main result, we explicitly give here the structure maps of  Schauenburg's Morita  base change (left) Hopf algebroids.
From now on, the unadorned symbol $\otimes$  stands for the tensor product over $\kk$, the commutative ground ring.

\subsection{Morita contexts}
\label{ssec: context}
Let $R$ and $S$ be two rings and let ${}_SP_R$ and ${}_RQ_S$ be two bimodules, together with the following bimodule isomorphisms:
\begin{equation}\label{Eq: Morita maps}
\begin{array}{rclrcl}
\phi: P \otimes_\erre Q &{\overset{\simeq}{\longrightarrow}}& S, & \phi^{-1}(1_\esse) &=& \sum p'_j \otimes_\erre q'_j, \\
\psi: Q \otimes_\esse P &{\overset{\simeq}{\longrightarrow}}& R, & \psi^{-1}(1_\erre) &=& \sum q_i \otimes_\esse p_i.
\end{array}
\end{equation}
It is known from Morita theory (see, e.g., \cite[p.\ 60]{Bas:AKT}) that, up to natural isomorphisms, $\phi$ and $\psi$ can be chosen in such a way that 
\begin{equation}
\label{energiaproxima}
(\phi\tensor{S}P) \,\,=\,\, (P\tensor{R}\psi) \quad \text{ and } \quad
(\psi\tensor{R}Q) \,\,=\,\, (Q\tensor{S}\phi).
\end{equation}
Thus $(R,S,P,Q,\phi,\psi)$ can be considered as a Morita context. In what follows, 
we will usually make use of the notation
$$
p'q' := \phi(p' \otimes_\erre q')  \quad \mbox{and} \quad qp := \psi(q \otimes_\esse p),\quad \forall \, p,p' \in P,\,\, q,q' \in Q.
$$
We then have 
$$
\sum_{j} p_j' \, q_j' \,\,=\,\, 1_S, \quad \quad \sum_{i} q_i \, p_i \,\,=\,\, 1_R,
$$
as well as
$$
a(bp) \,\,=\,\, (ab)p \quad \mbox{in} \quad {}_SP_R, \quad \quad
b(aq) \,\,=\,\, (ba) q \quad \mbox{in} \quad {}_RQ_S,
$$
for all pairs of elements $a, p \in P$ and $b,q \in Q$.

The above context is canonically extended to a Morita context between the enveloping rings $\Reh$ and $\Se$. 
That is, $(\Reh,\Se,\Pe,\Qe,\phi^{\rm e},\psi^{\rm e})$ is a Morita context as well, 
where the underlying bimodules are defined by 
$$
\begin{array}{rcl}
\Pe :=  P \otimes Q^\op &\in& \due \Mod \Se \Reh, \\
\Qe :=  Q \otimes P^\op &\in& \due \Mod \Reh \Se. 
\end{array}
$$
Here ${}_{R^\op}P^\op_{S^\op}$ and ${}_{S^\op}Q^\op_{R^\op}$ are the opposite 
bimodules, and $\phi^{\rm e},\psi^{\rm e}$ are the obvious maps. As was argued in \cite{Tak:squareMorita}, this is an induced $\sqrt{\mbox{Morita}}$ equivalence between $R$ and $S$, in the sense that the last context induces a monoidal equivalence between the monoidal categories of bimodules $ \due \Mod R R$ and $ \due \Mod S S$. 
Explicitly, such a monoidal equivalence is set up by the following functors
$$
\xymatrix@C=80pt{ \due \Mod R R \simeq {}_{\Reh}\Mod  \ar@<.5ex>[r]^-{\Pe\tensor{\Ree}-} &  {}_{\Se}\Mod \simeq \due \Mod S S.  \ar@<.5ex>[l]^-{\Qe\tensor{\Se}-} }
$$
One of the monoidal structure maps of the functor $\Qe\tensor{\Se}-$ is explicitly given by the following natural isomorphism
\begin{equation}\label{Eq:Phi}
\begin{array}{rcl}
\lr{\Qe\tensor{\Se}X}\tensor{R} \lr{\Qe\tensor{\Se}Y} &\overset{\simeq}{\lra}& \Qe\tensor{\Se}(X\tensor{S}Y), \\ 
\lr{(q\tensor{}p^o) \tensor{\Se} x} \tensor{R} \lr{(b\tensor{}a^o) \tensor{\Se} y}  &\lma& (q\tensor{}a^o) \tensor{\Se} \lr{x(pb)\tensor{S}y}, \\ 
\sum_j \lr{(q\tensor{}p'_j{}^o)\tensor{S}x}\tensor{R}\lr{(q_j'{}\tensor{}p^o)\tensor{S}y} &\longmapsfrom& (q\tensor{}p^o)\tensor{\Se}(x\tensor{S}y).
\end{array}
\end{equation}
An alternative way of defining these functors is via the following natural isomorphisms:
$$
\Qe\tensor{\Se}- \simeq Q\tensor{S}-\tensor{S}P, \qquad \Pe\tensor{\Ree}- \simeq P\tensor{R}-\tensor{R}Q.
$$
Repeating  the same process, we end up with two mutually inverse functors (up to natural isomorphisms)
$$
\xymatrix@C=100pt{ \due \Mod \Reh \Reh  \ar@<.5ex>[r]^-{\Pe\tensor{\Ree}(-)\tensor{\Ree}\Qe} &   \due \Mod \Se \Se.  \ar@<.5ex>[l]^-{\Qe\tensor{\Se}(-)\tensor{\Se}\Pe} }
$$
Using the Morita context, this equivalence is canonically lifted to the category of monoids. 
Thus, if we  denote by $\Rings{\Reh}$ the category of $\Reh$-rings, i.e., algebra extensions of $\Reh$, 
we have a commutative diagram
$$
\xymatrix@C=90pt@R=40pt{\Rings{\Reh}\ar@<.5ex>[rr]^-{\Pe\tensor{\Ree}(-)\tensor{\Ree}\Qe} \ar@{->}_-{\Scr{O}_R}[d] & &   \Rings{\Se}  \ar@{->}^-{\Scr{O}_S}[d] \ar@<.5ex>[ll]^-{\Qe\tensor{\Se}(-)\tensor{\Se}\Pe}  \\  \due \Mod \Reh \Reh  \ar@<.5ex>[rr]^-{\Pe\tensor{\Ree}(-)\tensor{\Ree}\Qe} & &   \due \Mod \Se \Se,  \ar@<.5ex>[ll]^-{\Qe\tensor{\Se}(-)\tensor{\Se}\Pe}  }
$$
whose vertical arrows are the forgetful functors.  For any $\Reh$-ring $T$ we then have  
functors connecting the categories of left modules:
\begin{equation}\label{Eq:modules}
\xymatrix@C=90pt@R=40pt{{}_T\Mod\ar@<.5ex>[rr]^-{\Pe\tensor{\Ree}(-)} \ar@{->}_-{\Scr{F}}[d] & &   {}_{\Pe\tensor{\Ree}T\tensor{\Ree}\Qe}\Mod  \ar@{->}^-{\Scr{F}'}[d] \ar@<.5ex>[ll]^-{\Qe\tensor{\Se}(-)}  \\   {}_{\Reh}\Mod   \ar@<.5ex>[rr]^-{\Pe\tensor{\Ree}(-)} & &   {}_{\Se}\Mod.  \ar@<.5ex>[ll]^-{\Qe\tensor{\Se}(-)}  }
\end{equation}

\subsection{Morita base change for left bialgebroids}
\label{subsect:Morita}
In \cite{Schau:MBCIQG}, Schauenburg used one of these functors to construct a functor 
from the category of left Hopf algebroids over $R$ to the category of left Hopf algebroids over $S$, known as 
\emph{Morita base change left Hopf algebroids}. 
In what follows, 
we will need an explicit description of this Morita base change left Hopf algebroid structure. So, it will be convenient to review this construction in more detail.

Let $(R,S,P,Q,\phi,\psi)$ be a Morita context. As one can realise from diagram \eqref{Eq:modules},  the following two assertions are equivalent: 
\begin{enumerate}
\item[(i)] the category of $T$-modules  is a monoidal category and the forgetful functor  $\Scr{F}$ is strict monoidal; 
\item[(ii)] the category of $(\Pe\tensor{\Ree}T\tensor{\Ree}\Qe)$-modules  is a monoidal category and the forgetful functor  $\Scr{F}'$ is strict monoidal.
\end{enumerate}
Therefore, by Schauenburg's result \cite[Theorem 5.1]{Schau:BONCRAASTFHB},  starting with a left Hopf algebroid $(R, U)$ we can construct a new one $(S, \tilde{U})$ as follows. Denote by 
$$
\tilde{U} := \Pe \otimes_\Ree U \otimes_\Ree \Qe
$$
the image of $U$. Using the natural isomorphism \eqref{Eq:Phi} and the diagram \eqref{Eq:modules} for the underlying $\Reh$-ring $U$, we can compute the structure maps of the left Hopf algebroid $(S,\tilde{U})$:
\begin{enumerate}
\item {\em Source and target.}
Source and target are given by 
\begin{equation}
\label{Eq:tilde_eta} 
\begin{array}{rcl}
\tilde{\eta}: \quad \Se &\lra & \Pe \otimes_\Ree U \otimes_\Ree \Qe, \\
s \otimes \td{s}^o  &\lma & \sum_{i,\, j} (sp'_j \otimes {q'_i}^o) \otimes_\Ree 1_{\uhhu} \otimes_\Ree (q'_j \otimes {(\td{s}p'_i)}^o).  
\end{array}
\end{equation}

\item {\em Ring structure.}
The multiplication in $\tilde{U}$ is given by
\begin{equation}
\label{Eq:tilde_mu}
\begin{array}{rcl}
\tilde{\mu}: \tilde{U} \otimes_\See \tilde{U} &\lra & \tilde{U}, \\
\td{u} \tensor{\See} \td{v}  &\lma & (a_1 \otimes b_1^o) \tensor{\Ree} \lr{\big(u \bract (c_1a_2)\big)\big((b_2d_1) \lact v \big) }
\tensor{\Ree} (c_2 \otimes d_2^o),
\end{array}
\end{equation}
where $\td{u}:= \lr{(a_1 \otimes b_1^o) \tensor{\Ree} u \tensor{\Ree} (c_1 \otimes d_1^o)}$ and $\td{v}:= \lr{(a_2 \otimes b_2^o) \tensor{\Ree} v \tensor{\Ree} (c_2 \otimes d_2^o)}$.
The identity element is given by the image $\td{\eta}(1_\See)$:
$$
1_\See \lma \sum (p'_j \otimes {q'_i}^o) \otimes_\Ree 1_{\uhhu} \otimes_\Ree (q'_j \otimes {p'_i}^o). 
$$
\item {\em Coring structure.}
The comultiplication is given by 
\begin{equation}
\label{Eq:tilde_Delta}
\begin{small}
\begin{array}{rcl}
\tilde{\gD}: \tilde{U} &\lra&  \tilde{U} \otimes_\esse \tilde{U}, \\
 \td{u}  &\lma& \underset{i,\, j}{\sum} \lr{(a \otimes q_i^o) \otimes_\Ree u_{(1)} \otimes_\Ree (c \otimes {p'_j}^o)}  
\otimes_{\esse} \lr{(p_i \otimes b^o) \otimes_\Ree u_{(2)} \otimes_\Ree (q'_j \otimes d^o)},
\end{array}
\end{small}
\end{equation}
where $\td{u}:= \lr{(a \otimes b^o) \tensor{\Ree} u \tensor{\Ree} (c \otimes d^o)}$,
and the counit is given by
\begin{equation}
\label{Eq:tilde_epsilon}
\tilde{\varepsilon}: \tilde{U} \longrightarrow S, \quad \td{u} \longmapsto a\varepsilon(u \bract (cd))b.
\end{equation}
\item {\em The left Hopf structure.}
The explicit expression for the translation map reads
\begin{equation}
\label{Eq:tilde_beta}
\begin{small}
\begin{array}{rcl}
\tilde{\gb}^{-1}: \tilde{U}  &\lra & \tilde{U} \otimes_\Sopp \tilde{U}, \\
\td{u}  &\lma & \underset{i,\, j}{\sum}\lr{(a \otimes {q'_j}^o) \otimes_\Ree u_+ \otimes_\Ree (c \otimes p_i^o)}
\otimes_{\Sopp} \lr{(d \otimes q_i^o) \otimes_\Ree u_- \otimes_\Ree (b \otimes {p'_j}^o)},  
\end{array}
\end{small}
\end{equation}
where again $\td{u}:= \lr{(a \otimes b^o) \tensor{\Ree} u \tensor{\Ree} (c \otimes d^o)}$.

\end{enumerate}

\subsection{$\tilde{U}$-modules and $\tilde{U}$-comodules}
\label{ssection:mod comod} 
Consider the diagram 
analogous to \eqref{Eq:modules} for right $U$-modules.
The functor of the first column in that diagram is explicitly given on objects as follows. 
For $M \in \Mod_{U}$, the right $\td{U}$-module $\tilde{M} := P \otimes_R M \otimes_R Q$ is equipped with the following action: 
denote 
$$
\tilde{m} := p\tensor{R}m\tensor{R}q \in \td{M}\quad \text{ and }\quad \td{u} := (a\tensor{}b^o)\tensor{\Ree}u\tensor{\Ree} (c\tensor{}d^o) \in \td{U},
$$
and define
\begin{equation}\label{Eq:act_tildeM}
\td{m} \td{u} := 
d \otimes_\erre \big((bp) \blact m \bract (qa)\big) u \otimes_\erre c.
\end{equation}
As shown in \cite{Schau:MBCIQG}, there is also a monoidal equivalence connecting  
the categories of left  comodules. 
More precisely,  if $M \in {}_U\bf Comod$, then $\tilde{M}$ is a left $\tilde{U}$-comodule with coaction
\begin{equation}\label{Eq:coact_tildeM}
\gD_{\tilde{\emme}}(\td{m}) := \sum_{i,\, j} \big((p \otimes q_i^o) \otimes_\Ree m_{(-1)} \otimes_\Ree (q \otimes {p'_j}^o)\big) \otimes_\esse (p_i \otimes_\erre m_{(0)} \otimes_\erre q'_j),
\end{equation}
which exactly coincides with the formula given in \cite{Schau:MBCIQG} in the special case where the left module $\due U \lact {}$ is finitely generated projective.

\begin{lemma}
\label{Lemma:I}
Let $M$ be a right $U$-module and left $U$-comodule. Then $M$ is aYD (resp. SaYD)
if and only if $\tilde{M}$ is.
\end{lemma}
\begin{proof}
It is sufficient to prove, say, the direct implication as the opposite direction then follows at once since both directions in the Morita base change induced equivalence between $U$-modules and $\td{U}$-modules as well as in the induced equivalence between $U$-comodules and $\td{U}$-comodules work the same way. 

So assume that $M$ is aYD.
Then, for any $s, t \in S$, we have 
\begin{eqnarray*}
 \td{m}\, \td{\eta}( t\tensor{}s^o) &=& \sum_{i,\, j}
 \Big(p\tensor{R}m\tensor{R}q \Big) \, \Big( ((t p'_j)\tensor{}q'_j{}^o) \tensor{\Ree}1_{U}
\tensor{\Ree} (q_j'\tensor{}(sp_i')^o) \Big) \\ &\overset{\eqref{Eq:act_tildeM}}{=}& 
\sum_{i,\,j} (sp_i'\tensor{R}\lr{(q_i'p)\blact m\bract (qtp_j')}\tensor{R} q_j' 
\\ &\overset{{\eqref{campanilla1},\eqref{campanilla2}}}{=}&  \sum_{i,\,j} (sp_i'\tensor{R}\lr{(q_i'p) m (qtp_j')}\tensor{R} q_j' \\ &=& \sum_{i,\,j} (sp_i'q_i'p)\tensor{R}m\tensor{R}(qtp_j'q_j') \\ &=& (sp)\tensor{R}m\tensor{R}(qt) \,\, =\,\, s\td{m}t,
\end{eqnarray*}
which gives \eqref{campanilla1} and \eqref{campanilla2} for $\td{M}$. 
Now, let us show  \eqref{huhomezone} for $\td{M}$, and
start with $\td{m}\td{u}\,=\, d\tensor{R}\big((bp)\blact m \ract(qa)\big)u\tensor{R} c$ as defined in \eqref{Eq:act_tildeM}. 
Once computed the coaction of the middle term in the latter tensor product 
and taking into account \eqref{huhomezone} for $M$,  apply \eqref{Eq:coact_tildeM} to obtain
\begin{footnotesize}
$$
\Delta_{\td{\emme}}(\td{m}\td{u}) = \sum_{i,\, j} \LR{(d\tensor{}q_i^o)\tensor{\Ree}\lr{(u_-\bract(bp))(m_{(-1)}\bract(qa))u_{+(1)}} \tensor{\Ree} (c\tensor{}p_j'{}^o)} \tensor{S} \lr{p_i\tensor{R}m_{(0)}u_{+(2)}\tensor{R}q_j'}.
$$
\end{footnotesize}
On the other hand, using \eqref{Eq:tilde_Delta} and \eqref{Eq:tilde_beta}, 
we get 
\begin{footnotesize}
\begin{equation*}
\begin{split}
\td{u}_- & \td{m}_{(-1)}\td{u}_{+(1)}\tensor{S}\td{m}_{(0)}\td{u}_{+(2)} \\
& = \sum_{i_0, i_1, i_2; \, j_0, j_1,j_2} \LR{ (d\tensor{}q_{i_1}^o)\tensor{\Ree}\Big(((q_{i_0}p_{j_1}')\blact u_-\bract(bp))((q_{i_2}p_{j_0}')\blact m_{(-1)} \bract (qa))u_{+(1)}\Big)\tensor{\Ree} (c\tensor{}p_{j_2}'{}^o)} \\ & \qquad \qquad
 \tensor{S}\LR{ p_{i_1}\tensor{R} \Big( (m_{(0)}\bract(q'_{j_0}p_{i_2}))((q'_{j_1}p_{i_0})\blact u_{+(2)})\Big)\tensor{R}q_{j_2}' } \\
& \!\! \overset{\eqref{douceuretresistance},\eqref{campanilla2}}{=}  
\sum_{i_0, i_1; \,  j_1,j_2} \LR{ (d\tensor{}q_{i_1}^o)\tensor{\Ree}\Big(((q_{i_0}p_{j_1}')\blact u_-\bract(bp))(m_{(-1)} \bract (qa))u_{+(1)}\Big)\tensor{\Ree} (c\tensor{}p_{j_2}'{}^o)} \\ & \qquad \qquad
 \tensor{S}\LR{ p_{i_1}\tensor{R} \Big( (m_{(0)})((q'_{j_1}p_{i_0})\blact u_{+(2)})\Big)\tensor{R}q_{j_2}' }
\\ 
& \!\! \overset{\eqref{panuelos}, \rmref{Sch3}}{=}  
\sum_{i_1; \, j_2} \LR{ (d\tensor{}q_{i_1}^o)\tensor{\Ree}\lr{(u_-\bract (bp))(m_{(-1)} \bract (qa))u_{+(1)}}\tensor{\Ree} (c\tensor{}p_{j_2}'{}^o)} \tensor{S}\LR{ p_{i_1}\tensor{R} \Big( m_{(0)}u_{+(2)}\Big)\tensor{R}q_{j_2}' } \\
&= \Delta_{\td{\emme}}(\td{m}\td{u}),  
\end{split}
\end{equation*}
\end{footnotesize}
where in the last equality we used \rmref{maotsetung} along with \rmref{campanilla1}--\rmref{huhomezone}. 
Analogously one checks the stability condition for $\td{M}$.
\end{proof}

\section{Morita base change invariance in Hopf-cyclic (co)homology}
This section contains our main results, Theorems \ref{thm:main1} \& \ref{thm:main2}. More precisely, 
we construct two morphisms between the cyclic modules $C_{\bull}(U, M)$ and $C_{\bull}(\tilde{U}, \tilde{M})$, 
where $(S, \tilde{U})$ is a Morita base change of $(R,U)$, and show that they form quasi-isomorphisms 
by giving an explicit homotopy. This establishes the Morita base change invariance for cyclic homology.  
For the Morita base change invariance of cyclic {\em co}homology, we follow the same path although we shall not give the proofs  
since they are similar to the homology case. 

Fix a Morita context $(R, S, P, Q, \phi, \psi)$ and assume we are given a left Hopf algebroid $(R, U)$, with 
Morita base change left Hopf algebroid  $(S, \tilde{U})$ as constructed in \S\ref{subsect:Morita}. 
Recall the notation of \S\ref{ssection:mod comod}, and from now on, 
the symbol  $i_{0,\ldots,n}$ stands for the set of indices $\{i_0,\cdots,i_n\}$.

\subsection{The homology case}
Consider the cyclic module $\big(C_\bull(U,M), d_\bull, s_\bull, t_\bull\big)$ as in \rmref{adualnightinpyongyang}.

\begin{lemma}
\label{Lemma:II}
Let  $M$ be a right $U$-module left
 $U$-comodule, subject to both \eqref{campanilla1} and \eqref{campanilla2}. 
Then the cyclic operator $\tilde{t}: C_n(\tilde{U}, \tilde{M}) \to C_n(\tilde{U}, \tilde{M})$ 
for the left Hopf algebroid $\tilde{U}$ with coefficients in $\td{M}$ is explicitly given by
\begin{multline*}
\tilde{t}: \tilde{m} \otimes_\Sopp \tilde{x} \lma \\ 
 \underset{i_{1,\ldots,n}}{\sum}\lr{p_{i_1}\tensor{R}m_{(0)}u^1_{+}\tensor{R}c_1}\tensor{S^\op}
\lr{(a_2\tensor{}q_{i_1}^o)\tensor{\Ree}u^2_{+}\tensor{\Ree} (c_2\tensor{}p_{i_2}^o)}\tensor{S^\op}  \\ \cdots\tensor{S^\op} \lr{(a_n\tensor{}q_{i_{n-1}}^o)\tensor{\Ree}u^n_{+}\tensor{\Ree}(c_n\tensor{}p_{i_n}^o)} \tensor{S^\op}  \\ \LR{(d_n\tensor{}q_{i_n}^o) \tensor{\Ree}\LR{(u^n_-\bract (b_nd_{n-1}))(u^{n-1}_-\bract (b_{n-1}d_{n-2}))\cdots (u^1_{-}\bract (b_1p)) \, m_{(-1)}}\tensor{\Ree}(q\tensor{}a_1^o)},
\end{multline*}
using the notation $\td{m} := p\tensor{R}m \tensor{R}q \in \td{M}$ as well as 
$\td{x}:= \td{u}^1\tensor{S^\op}\cdots\tensor{S^\op}\td{u}^{n}$, where $\td{u}^{k} := (a_k\tensor{}b_k^o)\tensor{\Ree}u^k\tensor{\Ree}(c_k\tensor{}d_k^o)$ for $1 \leq k \leq n$.
\end{lemma}
\begin{proof}
Eq.~\rmref{campanilla1} is not directly needed in the computation, but rather to make the operator $\tilde{t}$ well-defined.
By definition we know that 
$$
\td{t}(\td{m}\tensor{S^\op}\td{u}^1\tensor{S^\op}\cdots\tensor{S^\op}\td{u}^n)\,:=\,\td{m}_{(0)}\td{u}^1_{+} \tensor{S^\op}\td{u}^2_{+}\tensor{S^\op} \cdots \tensor{S^\op}\td{u}^n_{+}\tensor{S^\op} \lr {\td{u}^n_{-}\td{u}^{n-1}_- \cdots \td{u}^1_{-}\td{m}_{(-1)}}. 
$$ 
Using the formula for the translation map $\td{\beta}$ in \eqref{Eq:tilde_beta}, we have, along with Eqs.~\rmref{Eq:act_tildeM}, \rmref{Eq:coact_tildeM}, \rmref{salz}, and repeatedly using the multiplication formula \rmref{Eq:tilde_mu}
\begin{multline*}
\td{t}(\td{m}\tensor{S^\op}\td{u}^1\tensor{S^\op}\cdots\tensor{S^\op}\td{u}^n) \\ \,=\,
 \underset{\underset{i_{0,\ldots, n}}{j_{0,\ldots, n}}}{\sum} \Big(p_{i_1} \otimes_\erre \big((q'_{j_1}p_{i_0}) \blact m_{(0)} \bract (q'_{j_0}a_1)\big) u_{+}^1 \otimes_\erre c_1\Big)\tensor{S^\op}\lr{(a_2\tensor {}q_{j_2}'{}^o)\tensor{\Ree}u^2_{+}\tensor{\Ree}(c_2\tensor{}p_{i_2}^o)} \\ \tensor{S^\op} \cdots \tensor{S^\op}\lr{(a_{n}\tensor {}q_{j_{n}}'{}^o)\tensor{\Ree}u^{n}_{+}\tensor{\Ree}(c_{n}\tensor{}p_{i_n}^o)}  \tensor{S^\op} \left[\underset{}{} (d_{n}\tensor{}q_{i_n}^o)\tensor{\Ree} \right. \\ \Big[ \lr{(q_{i_{n-1}}p'_{j_n})\blact u^{n}_{-} \bract (b_nd_{n-1})} \lr{(q_{i_{n-2}}p'_{j_{n-1}})\blact u^{n-1}_{-} \bract (b_{n-1}d_{n-2})} \cdots \lr{(q_{i_{0}}p'_{j_{1}})\blact u^{1}_{-} \bract (b_{1}p)}m_{(-1)} \Big] \\ \tensor{\Ree}   
(q\tensor{}p'_{j_0}{}^o) \left. \underset{}{} \right] \\ \,=\,
 \underset{\underset{i_{0,\ldots, n}}{j_{0,\ldots, n}}}{\sum} \Big( p_{i_1}\tensor{R} \lr{(m_{(0)}\bract (q_{j_0}'a_1))(u_{+}^1 \ract (q_{j_1}'p_{i_0}))}\tensor{R}c_1\Big)\tensor{S^\op}\lr{(a_2\tensor {}q_{j_2}'{}^o)\tensor{\Ree}u^2_{+}\tensor{\Ree}(c_2\tensor{}p_{i_2}^o)} \\ \tensor{S^\op} \cdots \tensor{S^\op}\lr{(a_{n}\tensor {}q_{j_{n}}'{}^o)\tensor{\Ree}u^{n}_{+}\tensor{\Ree}(c_{n}\tensor{}p_{i_n}^o)}  \tensor{S^\op} \left[\underset{}{} (d_{n}\tensor{}q_{i_n}^o)\tensor{\Ree} \right. \\ \Big[ \lr{(q_{i_{n-1}}p'_{j_n})\blact u^{n}_{-} \bract (b_nd_{n-1})} \lr{(q_{i_{n-2}}p'_{j_{n-1}})\blact u^{n-1}_{-} \bract (b_{n-1}d_{n-2})} \cdots \lr{(q_{i_{0}}p'_{j_{1}})\blact u^{1}_{-} \bract (b_{1}p)}m_{(-1)} \Big] \\ \tensor{\Ree}   
(q\tensor{}p'_{j_0}{}^o) \left. \underset{}{} \right].
\end{multline*}
By Eqs.~\eqref{campanilla2} and \eqref{douceuretresistance}, we can eliminate the sum with the index $j_0$. Thus we have
\begin{multline*}
\td{t}(\td{m}\tensor{S^\op}\td{u}^1\tensor{S^\op}\cdots\tensor{S^\op}\td{u}^n) \\ \,=\,
 \underset{\underset{i_{0,\ldots,n}}{ j_{1,\ldots,n}}}{\sum} \Big( p_{i_1}\tensor{R} \lr{m_{(0)}(u_{+}^1 \ract (q_{j_1}'p_{i_0}))}\tensor{R}c_1\Big)\tensor{S^\op}\lr{(a_2\tensor {}q_{j_2}'{}^o)\tensor{\Ree}u^2_{+}\tensor{\Ree}(c_2\tensor{}p_{i_2}^o)} \\ \tensor{S^\op} \cdots \tensor{S^\op}\lr{(a_{n}\tensor {}q_{j_{n}}'{}^o)\tensor{\Ree}u^{n}_{+}\tensor{\Ree}(c_{n}\tensor{}p_{i_n}^o)}  \tensor{S^\op} \left[\underset{}{} (d_{n}\tensor{}q_{i_n}^o)\tensor{\Ree} \right. \\ \Big[ \lr{(q_{i_{n-1}}p'_{j_n})\blact u^{n}_{-} \bract (b_nd_{n-1})} \lr{(q_{i_{n-2}}p'_{j_{n-1}})\blact u^{n-1}_{-} \bract (b_{n-1}d_{n-2})} \cdots \lr{(q_{i_{0}}p'_{j_{1}})\blact u^{1}_{-} \bract (b_{1}p)}m_{(-1)} \Big] \\ \tensor{\Ree}   
(q\tensor{}a_1^o) \left. \underset{}{} \right].
\end{multline*}
Repeating the same process, but now using repeatedly \eqref{Sch3}, we can eliminate the sums indexed by $i_0,j_1,\cdots,j_n$, and obtain the stated formula. 
\end{proof}

In order to show invariance of Hopf-cyclic homology, we will first of all construct a quasi-isomorphism between the $b$-columns, 
denoted again by $C_\bull(U,M)$ resp.\ $C_\bull(\td{U},\td{M})$, 
of the cyclic bicomplexes $CC_{\bull\bull}(U,M)$ and $CC_{\bull\bull}(\td{U},\td{M})$ associated to the respective cyclic modules (cf.\ \S\ref{almeria}).

Define the map $\theta_n: C_n(U,M) \to C_n(\tilde{U},\tilde{M})$ as follows: for $n=0$, set
$$
\theta_0: M \longrightarrow \tilde{M}, \quad m \longmapsto \sum_i p_i \otimes_\erre m \otimes_\erre q_i,
$$
and for $n \geq 1$, abbreviating $x:= u^1 \otimes_\Ropp \cdots \otimes_\Ropp u^n$, set
\begin{equation}
\label{recattolico1}
\begin{split}
\theta_n: 
m \otimes_\Ropp x \longmapsto  \sum_{\scriptscriptstyle{i_{0, \ldots, {n-1}}} \atop \scriptscriptstyle{j_{0, \ldots, n}}} 
 & (p_{i_0} \otimes_\erre m \otimes_\erre q_{j_0}) \otimes_\Sopp \big((p_{j_0} \otimes q_{i_0}^o) \otimes_\Ree u^1 \otimes_\Ree (q_{j_1} \otimes p_{i_1}^o)\big) \\
&\quad \otimes_\Sopp  \cdots \otimes_\Sopp \big((p_{j_{n-1}} \otimes q_{i_{n-1}}^o) \otimes_\Ree u^{n} \otimes_\Ree (q_{j_{n}} \otimes p_{j_{n}}^o)\big). 
\end{split}
\end{equation}
In the opposite direction, introduce the map $\gamma_n: C_n(\tilde{U},\tilde{M}) \to  C_n(U,M)$, which is, for $n=0$,
$$
\gamma_0: \tilde{M} \longrightarrow M, \quad \big( \td{m}:=p \otimes_\erre m \otimes_\erre q\big) \longmapsto \sum_{j}(q'_jp)m(qp'_j),
$$
and for $n \geq 1$ it is given as
\begin{equation}
\label{reinacattolica1}
\begin{footnotesize}
{
\begin{split}
& \gamma_n:  \tilde{m} \otimes_\Sopp \tilde{x} \longmapsto  \sum_{j_{0, \ldots, n}} 
 m(q p'_{j_0}) \otimes_\Ropp \lr{ (q'_{j_0}a_1) \lact u^1 \ract (b_1p) \bract (c_1p'_{j_1}) } \\
& \quad  \otimes_\Ropp \lr{(q'_{j_1}a_2) \lact u^2 \ract (b_2d_1) \bract (c_2p'_{j_2})} \otimes_\Ropp 
\cdots \otimes_\Ropp \lr{(q'_{j_n} d_n) \blact (q'_{j_{n-1}}a_n) \lact  u^n \ract (b_nd_{n-1}) \bract (c_np'_{j_n})},
\end{split} }
\end{footnotesize}
\end{equation}
where $\tilde{u}^i := (a_i \otimes b^o_i) \otimes_\Ree u^i \otimes_\Ree (c_i \otimes d^o_i) \,\in \td{U}$ 
for $1 \leq i \leq n$, and
$\tilde{x}:= \tilde{u}^1 \otimes_\Sopp \cdots \otimes_\Sopp \tilde{u}^n$.

\begin{lem}
\label{revell}
The maps $\theta_{\bull}$ and $\gamma_{\bull}$ are morphisms of chain complexes. 
\end{lem}

\begin{proof}
We only check the compatibility of the differential with $\gamma_n$ since the computation for $\theta_n$ is similar but less complicated. 
Decompose
$$
b\gamma_n\,\,=\,\, \overset{=:(i)}{\overbrace{d_0\gamma_n}} \,+\, \overset{=:(ii)}{\overbrace{\sum_{1\leq \, k\,  \leq n-1} (-1)^k d_k \gamma_n}} \,+\, \overset{=:(iii)}{\overbrace{(-1)^n d_n\gamma_n}},
$$ 
where $b$ is the differential \rmref{Eq:b} of the underlying simplicial structure of $C_{\bull}({U}, {M})$ as in \eqref{adualnightinpyongyang}. 
When applying this map to an element of the form 
$\td{m}\tensor{S^\op}\td{u}^1\tensor{S^\op}\cdots\tensor{\Sopp}\td{u}^n$ 
(using the notation above), each term is explicitly given by 
\begin{footnotesize}
\begin{eqnarray*} 
(i) &\!\!\!\!\!\!=&\!\!\!\!\!\! \underset{j_{0,\, \ldots,n}}{\sum}  (m(qp_{j_0}') \tensor{\Ree}\Big( \lr{(q_{j_0}'a_1)\lact u^1\ract(b_1p)}\bract (c_1p_{j_1}')\Big) \tensor{\Ree}\cdots  \tensor{\Ree} 
\\ 
&\!\!\!\!\!\!\!\!\!\! &\!\!\!\!
\!\!\!\!\!\! \LR{ \varepsilon\Big( (q_{j_n}'d_n)\blact \lr{(q_{j_{n-1}}'a_n)\lact u^n\ract(b_nd_{n-1})}\bract(c_np_{j_n}')\Big) \blact \Big( \lr{(q_{j_{n-2}}'a_{n-1})\lact u^{n-1}\ract(b_{n-1}d_{n-2})}\bract(c_{n-1}p_{j_{n-1}}') \Big)} 
\\
\!\!\!\!
 &\!\!\!\!\!\!\overset{\eqref{Eq:counit}}{=} &\!\!\!\!\!\!
  \underset{j_{0,\, \ldots,n-1}}{\sum}  (m(qp_{j_0}') \tensor{\Ree}\Big( \lr{(q_{j_0}'a_1)\lact u^1\ract(b_1p)}\bract (c_1p_{j_1}')\Big) \tensor{\Ree}\cdots  \tensor{\Ree} \\ & & \LR{\Big((q_{j_{n-1}}'a_{n}) \varepsilon( u^n\bract(c_nd_n))(b_nd_{n-1})\Big) \blact \Big( \lr{(q_{j_{n-2}}'a_{n-1})\lact u^{n-1}\ract(b_{n-1}d_{n-2})}\bract(c_{n-1}p_{j_{n-1}}') \Big)};
\\
(ii) &\!\!\!\!\!\!=&\!\!\!\!\!\!
 \sum_{k=1}^{n-1}\underset{j_{0,\ldots,{n-k-1},{n-k+1},\ldots,{n}}}{\sum}(m(qp_{j_0}')\tensor{\Ree}\Big( \lr{(q_{j_0}'a_1)\lact u^1\ract(b_1p)}\bract (c_1p_{j_1}')\Big)   \tensor{\Ree}\cdots \tensor{\Ree} \\ 
&\!\!\!\!\!\!\!\!\!\!\!&\!\!\!\!\!\!\!\!\!\!\! \LR{  \Big( (b_{n-k-1}d_{n-k})\blact \lr{(q_{j_{n-k-1}}'a_{n-k})\lact u^{n-k} \ract (b_{n-k}d_{n-k-1})}\bract (c_{n-k}a_{n-k+1})\Big)   \Big( u^{n-k+1} \bract(c_{n-k+1}p_{j_{n-k+1}}')\Big) } \\
&\!\!\!\!\!\!&\!\!\!\!\!\! \qquad \qquad 
\tensor{\Ree}\cdots \tensor{\Ree} 
\Big((q_{j_n}'d_n) \blact \lr{(q_{j_{n-1}}'a_n)\lact u^{n}\ract(b_nd_{n-1})}\bract (c_np_{j_n}')\Big); \\
(iii) &\!\!\!\!\!\!=&\!\!\!\!\!\!
  \sum_{j_{1, \ldots,n}} (-1)^n \Big( m\lr{\lr{(qa_1)\lact u^1\ract(b_1p)}\bract (c_1p_{j_1}') }\Big)\tensor{\Ree} 
  \\
&& \qquad \qquad \qquad \qquad \qquad \qquad \qquad 
\cdots \tensor{\Ree} \Big((q_{j_n}'d_n) \blact \lr{(q_{j_{n-1}}'a_n)\lact u^{n}\ract(b_nd_{n-1})}\bract (c_np_{j_n}')\Big). 
\end{eqnarray*}
\end{footnotesize}
On the other hand, we can also write  
$$ 
\gamma_{n-1}\td{b}\,\,=\,\, \overset{=:\widetilde{(i)}}{\overbrace{\gamma_{n-1} \td{d}_0}} \,+\, \overset{=:\widetilde{(ii)}}{\overbrace{\sum_{1\leq \, k\,  \leq n-1} (-1)^k  \gamma_{n-1}\td{d}_k}} \,+\, \overset{=:\widetilde{(iii)}}{\overbrace{(-1)^n \gamma_{n-1}\td{d}_n}},
$$
where $\td{b}$ is analogously the differential of the underlying simplicial structure of $C_{\bull}(\td{U}, \td{M})$.
Applying $\gamma_{n-1} \td{b}$ to the same element $\td{m}\tensor{S^\op}\td{u}^1\tensor{S^\op}\cdots\tensor{\Sopp}\td{u}^n$, 
we find that the first term is
\begin{footnotesize}
\begin{eqnarray*}
\widetilde{(i)} &=& \underset{j_{0, \ldots,{n-1}}}{\sum}  m(qp_{j_0}') \tensor{\Ree}\Big( \lr{(q_{j_0}'a_1)\lact u^1\ract(b_1p)}\bract (c_1p_{j_1}')\Big) \tensor{\Ree}\\
&& \quad \cdots \tensor{\Ree} \Big(  (q_{j_{n-1}}'\hat{d}_{n-1})\blact \lr{ (q_{j_{n-2}}'\hat{a}_{n-1})\lact \hat{u}^{n-1}\ract(\hat{b}_{n-1}d_{n-2}) } \bract(\hat{c}_{n-1}p_{j_{n-1}}') \Big),
\end{eqnarray*}
\end{footnotesize}
where we denoted the elements $\td{\varepsilon}(\td{u}^n)\blact \td{u}^{n-1} =: 
(\hat{a}_{n-1}\tensor{}\hat{b}_{n-1}{}^o)\tensor{\Ree}\hat{u}^{n-1}\tensor{\Ree}(\hat{c}_{n-1}\tensor{}\hat{d}_{n-1}{}^o)$. 
Computing explicitly this term, we obtain
\begin{footnotesize}
\begin{eqnarray*}
\td{\varepsilon}(\td{u}^n)\blact \td{u}^{n-1}&\!\!\!\!\!\!
=&\!\!\!\!\!\! \td{u}^{n-1}\td{\eta}(1\tensor{}\td{\varepsilon}(\td{u}^n)^o) \nonumber \\
&\!\!\!\!\!\!\overset{\eqref{Eq:tilde_eta} }{=}&\!\!\!\!\!\! \sum_{i_n,\, j_n} 
({a}_{n-1}\tensor{}{b}_{n-1}{}^o)\tensor{\Ree}\Big ((q_{i_n}'d_{n-1})\blact{u}^{n-1} \bract (c_{n-1}p_{j_n}') \Big) \tensor{\Ree}(q_{j_n}'\tensor{}(\td{\varepsilon}(\td{u}^n) p_{i_n}'){}^o) \nonumber \\
&\!\!\!\!\!\!\overset{\eqref{Eq:tilde_epsilon}}{=}&\!\!\!\!\!\!
\sum_{i_n,\, j_n} 
({a}_{n-1}\tensor{}{b}_{n-1}{}^o)\tensor{\Ree}\Big ((q_{i_n}'d_{n-1})\blact{u}^{n-1} \bract (c_{n-1}p_{j_n}') \Big)  \tensor{\Ree}(q_{j_n}'\tensor{}\lr{a_n \varepsilon(u^n\bract (c_nd_n)) b_n p_{i_n}'){}^o}  \nonumber\\ 
&\!\!\!\!\!\!=&\!\!\!\!\!\! 
\sum_{ j_n} 
({a}_{n-1}\tensor{}{b}_{n-1}{}^o)\tensor{\Ree}\Big ({u}^{n-1} \bract (c_{n-1}p_{j_n}') \Big)  \tensor{\Ree}(q_{j_n}'\tensor{}\lr{a_n \varepsilon(u^n\bract (c_nd_n)) b_n d_{n-1}){}^o}
\nonumber\\ 
&\!\!\!\!\!\!=&\!\!\!\!\!\! 
({a}_{n-1}\tensor{}{b}_{n-1}{}^o)\tensor{\Ree}{u}^{n-1}   \tensor{\Ree}\Big(\big( c_{n-1}(\sum_{ j_n} p_{j_n}'q_{j_n}') \big)\tensor{}\lr{a_n \varepsilon(u^n\bract (c_nd_n)) b_n d_{n-1}){}^o} \Big) \nonumber \\ 
&\!\!\!\!\!\!= &\!\!\!\!\!\!
({a}_{n-1}\tensor{}{b}_{n-1}{}^o)\tensor{\Ree}\, {u}^{n-1}  \, \tensor{\Ree} \Big(c_{n-1}\tensor{}\lr{a_n \varepsilon(u^n\bract (c_nd_n)) b_n d_{n-1}){}^o}\Big),
 \label{Eq:T1}
\end{eqnarray*}
\end{footnotesize}
thence, 
$$ 
\hat{a}_{n-1}=a_{n-1},\quad \hat{b}_{n-1}= {b}_{n-1}, \quad \hat{u}^{n-1}=u^{n-1},\quad \text{ and } \quad \hat{d}_{n-1}=a_n \varepsilon(u^n\bract (c_nd_n)) b_n d_{n-1}.
$$
Inserting this into the expression of $\widetilde{(i)}$ above, one obtains $\widetilde{(i)}=(i)$.
The second term can be written as follows:
\begin{footnotesize}
\begin{multline*} 
\widetilde{(ii)} \,=\, \sum_{k=1}^{n-1}\underset{j_{0,\ldots,{n-k-1},{n-k+1},\ldots,{n}}}{\sum}(m(qp_{j_0}')\tensor{\Ree}\Big( \lr{(q_{j_0}'a_1)\lact u^1\ract(b_1p)}\bract (c_1p_{j_1}')\Big)   \tensor{\Ree}\cdots \\ \tensor{\Ree} 
 \Big( \lr{(q_{j_{n-k-1}}'\bara{a}_{n-k})\lact \bara{u}^{n-k} \ract (\bara{b}_{n-k}d_{n-k-1})}\bract (\bara{c}_{n-k}p_{j_{n-k+1}}')\Big) \\   \qquad \qquad \qquad \tensor{\Ree}
 \Big( \lr{(q_{j_{n-k+1}}'a_{n-k+2})\lact u^{n-k+2} \ract (b_{n-k+2}\bara{d}_{n-k}) }\bract(c_{n-k+2}p_{j_{n-k+1}}')\Big)  \\
\tensor{\Ree}\cdots \tensor{\Ree} \Big((q_{j_n}'d_n) \blact \lr{(q_{j_{n-1}}'a_n)\lact u^{n}\ract(b_nd_{n-1})}\bract (c_np_{j_n}')\Big), 
\end{multline*}
\end{footnotesize}
where we denoted the elements 
\begin{footnotesize}
\begin{eqnarray} 
 \td{u}^{n-k} \td{u}^{n-k+1} &\!\!\!\!\!\!:=&\!\!\!\!\!\! (\bara{a}_{n-k}\tensor{}\bara{b}_{n-k}{}^o)\tensor{\Ree}\bara{u}^{n-k}\tensor{\Ree}(\bara{c}_{n-k}\tensor{}\bara{d}_{n-k}{}^o) 
\nonumber \\
\label{Eq:T2}
&\!\!\!\!\!\!\overset{\eqref{Eq:tilde_mu}}{\,=}&\!\!\!\!\!\!  (a_{n-k}\tensor{}b_{n-k}^o)\tensor{\Ree} \lr{(b_{n-k+1}d_{n-k}) \blact u^{n-k} \bract (c_{n-k}a_{n-k+1})}  u^{n-k+1} \\
\nonumber
&& \hspace*{8cm} \tensor{\Ree}(c_{n-k+1}\tensor{}d_{n-k+1}^o). 
\end{eqnarray}
\end{footnotesize}
Therefore, $\widetilde{(ii)} = (ii)$ after substituting 
\eqref{Eq:T2} in $\widetilde{(ii)}$.
As for the third term, we have
\begin{footnotesize}
\begin{eqnarray*}
\widetilde{(iii)} &=&  \sum_{j_{1, \ldots,n}} (-1)^n \Big( m\lr{\lr{(qa_1)\lact u^1\ract(b_1p)}\bract (c_1p_{j_1}') }\Big)\tensor{\Ree}\hspace{3cm} \\ & & \hspace{2cm} \cdots \tensor{\Ree} \Big((q_{j_n}'d_n) \blact \lr{(q_{j_{n-1}}'a_n)\lact u^{n}\ract(b_nd_{n-1})}\bract (c_np_{j_n}')\Big),
\end{eqnarray*}
\end{footnotesize}
which is obviously $(iii)$. We conclude that $\gamma_{\bull}$ is a morphism of chain complexes. 
\end{proof}

\begin{proposition}\label{prop: h}
The composite $\gamma_n \theta_n$ is homotopic to the identity, the homotopy $h_n: C_n(U, M) \to C_{n+1}(U,M)$ 
being explicitly given by the following map: for $n=0$, define  
$$
h_0: m \mapsto \sum_{i,\,j} m(q_ip'_j) \otimes_\Ropp ((q'_jp_i) \lact 1_\uhhu),
$$
and for $n\geq 1$, set
\begin{equation}\label{reyescattolicos}
\begin{footnotesize}
{
\begin{split}
h_n: m \otimes_\Ropp x \mapsto 
& \sum^n_{k=0}  \underset{\underset{i_{0, \ldots, k}}{j_{0,\ldots, k}}}{\sum}  (-1)^{k+n} m(q_{i_0}p'_{j_0}) 
\otimes_\Ropp \lr{ (q'_{j_0}p_{i_0}) \lact u^1 \bract (q_{i_1}p'_{j_1}) }
\otimes_\Ropp  \\
& \quad \cdots  \otimes_\Ropp\lr{ (q'_{j_{k-1}}p_{i_{k-1}}) \lact u^k \bract   (q_{i_k}p'_{j_k}) }   \otimes_\Ropp  \Big( (q'_{j_k}p_{i_k})  \lact 1_\uhhu \Big) \otimes_\Ropp u^{k+1} \otimes_\Ropp \cdots \otimes_\Ropp u^n 
\end{split} }
\end{footnotesize}
\end{equation}
abbreviating $x:=u^1 \otimes_\Ropp \cdots \otimes_\Ropp u^n$ as before. 
Similarly, $\theta_n \gamma_n$ is homotopic to the identity as well.
\end{proposition}

\begin{proof}
We need to check $bh_0 = \gamma_0 \theta_0 - \id$ for $n=0$ and $bh_n + h_{n-1}b = \gamma_n \theta_n - \id$ for $n > 0$. 
As for the first one, it is immediate that
\begin{equation*}
b h_0(m) = \sum_{i,j} \varepsilon((q'_jp_i) \lact 1_\uhhu) \blact m(q_ip'_j) - m \\
= \sum_{i, j} (q'_jp_i)\big(m) (q_ip'_j) - m
= \gamma_0 \theta_0(m) - m.
\end{equation*}
In case $n > 0$, since multiplying two consecutive tensor factors of $h_n$ kills the respective $q, p$ as well as the $q', p'$ 
between them, 
it is
straightforward to see that 
\begin{equation}
\label{ronda}
\sum_{k=1}^n (-1)^k d_k h_n(m \otimes_\Ropp x) + \sum_{k=1}^{n-1} (-1)^k h_{n-1} d_k (m \otimes_\Ropp x) =0. 
\end{equation}
As for the remaining terms, we have
\begin{footnotesize}
\begin{equation*}
\begin{split}
&(-1)^{n+1} d_{n+1}   h_n(m \otimes_\Ropp x) \\ 
&= - m \otimes_\Ropp x 
+  (-1)^{n+1} \sum^n_{k=1} \underset{\underset{i_{1, \ldots, k}}{j_{1,\ldots, k}}}{\sum} (-1)^{k+n} mu^1(q_{i_1}p'_{j_1}) 
\otimes_\Ropp  \cdots  \otimes_\Ropp\lr{ (q'_{j_{k-1}}p_{i_{k-1}}) \lact u^k \bract   (q_{i_k}p'_{j_k}) } \\
& \hspace*{8cm} \otimes_\Ropp  \Big( (q'_{j_k}p_{i_k})  \lact 1_\uhhu \Big) \otimes_\Ropp u^{k+1} \otimes_\Ropp \cdots \otimes_\Ropp u^n \\
& = - m \otimes_\Ropp x 
+  (-1)^{n+1} \sum^{n-1}_{k=0} \underset{\underset{i_{0, \ldots, k}}{j_{0,\ldots, k}}}{\sum} (-1)^{k+(n-1)} (mu^1)(q_{i_0}p'_{j_0}) 
\otimes_\Ropp  \lr{ (q'_{j_{0}}p_{i_{0}}) \lact u^{2} \bract   (q_{i_1}p'_{j_1}) } \otimes_\Ropp  \cdots \\
&  \hspace*{3.5cm} \otimes_\Ropp \lr{ (q'_{j_{k-1}}p_{i_{k-1}}) \lact u^{k+1} \bract   (q_{i_k}p'_{j_k}) } \otimes_\Ropp  \Big( (q'_{j_k}p_{i_k})  \lact 1_\uhhu \Big) \otimes_\Ropp u^{k+2} \otimes_\Ropp \cdots \otimes_\Ropp u^n \\
&= (- \id -  (-1)^n h_{n-1} d_n) (m \otimes_\Ropp x).
\end{split}
\end{equation*}
\end{footnotesize}
Moreover,
\begin{footnotesize}
\begin{equation*}
\begin{split}
& d_{0} h_n(m \otimes_\Ropp x)  \\ 
& =\sum^{n-1}_{k=0} \underset{\underset{i_{0, \ldots, k}}{j_{0,\ldots, k}}}{\sum}  (-1)^{k+n} (m \bract (q_{i_0}p'_{j_0}) )
\otimes_\Ropp \lr{ (q'_{j_0}p_{i_0}) \lact u^1 \bract (q_{i_1}p'_{j_1}) }
\otimes_\Ropp  \cdots  \\
&\qquad \quad \otimes_\Ropp\lr{ (q'_{j_{k-1}}p_{i_{k-1}}) \lact u^k \bract   (q_{i_k}p'_{j_k}) }   \otimes_\Ropp  \Big( (q'_{j_k}p_{i_k})  \lact 1_\uhhu \Big) \otimes_\Ropp u^{k+1} \otimes_\Ropp \cdots \otimes_\Ropp u^{n-2} \otimes_\Ropp \big(\gve(u^n) \blact u^{n-1}\big) \\
& 
\quad +  \underset{\underset{i_{0,\ldots,{n}}}{j_{0,\ldots,{n}}}}{\sum} (m \bract (q_{i_0}p'_{j_0}) )
\otimes_\Ropp \lr{ (q'_{j_0}p_{i_0}) \lact u^1 \bract (q_{i_1}p'_{j_1}) }
\otimes_\Ropp  \cdots \otimes_\Ropp 
\big((q'_{j_{n-1}}p_{i_{n-1}}) \lact (q'_{j_n}p_{i_n}) \blact u^n \bract   (q_{i_n}p'_{j_n}) \big), 
\end{split}
\end{equation*}
\end{footnotesize}
where we have used Eqs.~\rmref{pouvoir} and \rmref{Eq:counit}. 
The first sumand is easily seen to be equal to $- h_{n-1} d_0$ and we are left with computing the last sumand: 
by definition of $\theta_n$ and $\gamma_n$ (see Eqs.~\eqref{recattolico1}--\eqref{reinacattolica1})
\begin{footnotesize}
\begin{eqnarray*}
 \gamma_n\theta_n(m\otimes_\Ropp x) &=& \gamma_n\Big[
\sum_{\scriptscriptstyle{k_{0, \ldots, {n-1}}} \atop \scriptscriptstyle{j_{0, \ldots, n}}} 
  (p_{k_0} \otimes_\erre m \otimes_\erre q_{j_0}) \otimes_\Sopp \big((p_{j_0} \otimes q_{k_0}^o) \otimes_\Ree u^1 \otimes_\Ree (q_{j_1} \otimes p_{k_1}^o)\big)  \\
& &  \qqquad \qqquad 
\otimes_\Sopp  \cdots \otimes_\Sopp 
\big((p_{j_{n-1}} \otimes q_{k_{n-1}}^o) \otimes_\Ree u^{n} \otimes_\Ree (q_{j_{n}} \otimes p_{j_{n}}^o)\big)\underset{\underset{}{}}{\underset{}{}} \Big]
\\ &=& 
 \sum_{\scriptscriptstyle{k_{0, \ldots, {n-1}}}  \atop \scriptscriptstyle{j_{0, \ldots, n};\, i_{0,\ldots,n}}} 
    m (q_{j_0}p_{i_0}') \otimes_\Rop \big((q_{i_0}'p_{j_0})\lact  u^1 \ract (q_{k_0}p_{k_0}) \bract (q_{j_1}p'_{i_1}) \big) 
\\ 
&& \qquad \qquad \otimes_\Rop \cdots \otimes_\Rop 
\big((q'_{i_{n}}p_{j_n})\blact (q_{i_{n-1}}'p_{j_{n-1}}) \lact  u^{n} \ract (q_{k_{n-1}} p_{k_{n-1}}) \bract(q_{j_{n}} p_{i_{n}}')\big) \\
&= & 
\quad   \underset{\underset{i_{0,\ldots,{n}}}{j_{0,\ldots,{n}}}}{\sum} (m \bract (q_{i_0}p'_{j_0}) )
\otimes_\Ropp \lr{ (q'_{j_0}p_{i_0}) \lact u^1 \bract (q_{i_1}p'_{j_1}) }
\\ & &  
\qquad \qquad
\otimes_\Ropp  \cdots \otimes_\Ropp 
\big((q'_{j_{n-1}}p_{i_{n-1}}) \lact (q'_{j_n}p_{i_n}) \blact u^n \bract   (q_{i_n}p'_{j_n}) \big),
\end{eqnarray*}
\end{footnotesize}
where \rmref{campanilla2} was used in the last line and which, as is seen by interchanging the indices, is exactly the last term in the expression of $d_{0} h_n(m \otimes_\Ropp x)$ above. Hence we have shown that 
$$
d_{0} h_n(m \otimes_\Ropp x) \,  =\, \big(-h_{n-1} d_0 +  \gamma_n \theta_n\big)(m \otimes_\Ropp x). 
$$
Combining this with \rmref{ronda}, we obtain
$
bh_n + h_{n-1}b = \gamma_n \theta_n - \id,
$
and this finishes the proof.
\end{proof}

To pass to the cyclic case, we prove first:
\begin{lemma}
\label{Lemma:tgamma}
The morphisms of chain complexes $\theta_{\bull}$ and $\gamma_{\bull}$ are morphisms of cyclic objects. That is, they satisfy:
$$ \gamma_{\bull} \, \td{t}_{\bull}\,\,=\,\, t_{\bull} \, \gamma_{\bull},\qquad 
\theta_{\bull} \, t_{\bull}\,\,=\,\, \td{t}_{\bull} \, \theta_{\bull}.$$
\end{lemma}
\begin{proof}
We only check the first equation. 
Take an element 
$\td{m}\tensor{\Sopp}\td{u}^1 \tensor{\Sopp}\cdots\tensor{\Sopp} \td{u}^n \in C_n(\td{U},\, \td{M})$, for $n \geq 0$. Then, applying equations \eqref{Sch47}, \eqref{Sch5}, and \eqref{maotsetung}, we can write
\begin{footnotesize}
\begin{equation*}
\begin{split}
t_n \gamma_n(\td{m}\tensor{\Sopp}\td{u}^1 \tensor{\Sopp}\cdots\tensor{\Sopp} \td{u}^n)
 &=  \sum_{j_{0,\ldots,n}} \Big( \lr{m_{(0)}\bract (q_{j_0}'a_1)}\lr{u_{+}^1 \bract (c_1p_{j_1}')}\Big) \tensor{\Ropp} \Big( (q_{j_1}'a_2)\lact u_{+}^2 \bract (c_2p_{j_2}')\Big) \\ 
& \qquad
\tensor{\Ropp} \cdots\tensor{\Ropp}\Big((q_{j_{n-1}}'a_n) \lact u_{+}^n \bract (c_np_{j_n}')\Big) \\ 
& \qquad
  \tensor{\Ropp}\left[\underset{}{} ((q_{j_n}'d_n) \lact u^n_{-})((b_nd_{n-1}) \lact u^{n-1}_{-}) ((b_{n-1}d_{n-2}) \lact u^{n-2}_{-}) \cdots \right. \\ 
& \qquad 
\cdots \left. \underset{}{} ( (b_2d_1)\lact u^1_{-}) ( (b_1p) \lact m_{(-1)}\bract(qp_{j_0}')) \right].
\end{split}
\end{equation*}
\end{footnotesize}
On the other hand, we have 
\begin{footnotesize}
\begin{eqnarray*}
\gamma_n\td{t}_n(\td{m}\tensor{\Sopp}\td{u}^1 \tensor{\Sopp}\cdots\tensor{\Sopp} \td{u}^n)
 \!\!\!\!&=\!\!\!\!&  \sum_{j_{0,\ldots,n}} \Big( \lr{m_{(0)}u_{+}^1} \bract (c_1p_{j_0}')\Big) \tensor{\Ropp} \Big( (q_{j_0}'a_2)\lact u_{+}^2 \bract (c_2p_{j_1}')\Big) \\
&& \qquad 
\tensor{\Ropp} \cdots\tensor{\Ropp}\Big( (q_{j_{n-2}}'a_n) \lact u_{+}^n \bract (c_np_{j_{n-1}}')\Big) \\
&& \qquad   
\tensor{\Ropp}\left[\underset{}{} ((q_{j_{n-1}}'d_n)\lact u^n_{-})((b_nd_{n-1}) \lact u^{n-1}_{-}) ((b_{n-1}d_{n-2}) \lact u^{n-2}_{-}) \cdots \right. \\
&& \qquad\qquad  
\cdots \left.\underset{}{} ((b_2d_1) \lact u^1_{-}) \Big( (b_1p) \lact \lr{ (q_{j_n}'a_1) \blact m_{(-1)}\bract(qp_{j_n}')} \Big) \right] \\
\!\!\!\!
&\overset{\eqref{douceuretresistance}}{=}\!\!\!\!&  \sum_{j_{0,\ldots,n}} \Big( \lr{(m_{(0)}\bract(q_{j_n}'a_1))u_{+}^1} \bract (c_1p_{j_0}')\Big) \tensor{\Ropp} \Big( (q_{j_0}'a_2)\lact u_{+}^2 \bract (c_2p_{j_1}')\Big) \\ 
&& \qquad  
\tensor{\Ropp} \cdots\tensor{\Ropp}\Big( (q_{j_{n-2}}'a_n) \lact u_{+}^n \bract (c_np_{j_{n-1}}')\Big) \\ 
&& \qquad 
\tensor{\Ropp}\left[\underset{}{} ((q_{j_{n-1}}'d_n)\lact u^n_{-})((b_nd_{n-1}) \lact u^{n-1}_{-}) ((b_{n-1}d_{n-2}) \lact u^{n-2}_{-}) \cdots \right. \\ 
&& \qquad \qquad
\cdots \left.\underset{}{} ((b_2d_1) \lact u^1_{-}) \Big( (b_1p) \lact \lr{ m_{(-1)} \bract (qp_{j_n}') } \Big) \right].
\end{eqnarray*}
\end{footnotesize}
Now, renumbering the indices we find the equality, and this finishes the proof.
\end{proof}

Combining Lemma \ref{revell}, Proposition \ref{prop: h}, and Lemma \ref{Lemma:tgamma}, 
we conclude that $\theta_\bull$ and $\gamma_\bull$ are in particular {\em equivalences} of cyclic modules. Consequently, we 
can now formulate the main theorem of this paper:
\begin{thm}
\label{thm:main1}
{\rm (Morita base change invariance of (Hopf-)cyclic homology).}
Let $(R, U)$ be a left Hopf algebroid, $M$ a left $U$-comodule right $U$-module which is SaYD (i.e., satisfies \eqref{campanilla1}--\eqref{Eq:SaYD}), and $(R,S,P,Q,\phi,\psi)$ a Morita context. We then have the following natural $\Bbbk$-module isomorphisms:
\begin{eqnarray*}
H_\bull(U, M) &\simeq& H_\bull(\td{U}, P \otimes_\erre M \otimes_\erre Q), \\
HC_\bull(U, M) &\simeq& HC_\bull(\td{U}, P \otimes_\erre M \otimes_\erre Q). 
\end{eqnarray*}
\end{thm}
\begin{proof}
This follows at once by using the $SBI$ sequence for cyclic modules, cf.~\cite[\S2.5.12]{Lod:CH} for details.
\end{proof}

\subsection{The cohomology case}

In this section, we will consider the case of Hopf-cyclic cohomology under Morita base change. Since all steps are basically analogous to the preceding section, we refrain from spelling out the details and just indicate the main ingredients. 

Consider the cocyclic module 
$\big(C^\bull(U,M), \gd_\bull, \gs_\bull, \tau_\bull\big)$ 
as in \rmref{anightinpyongyang}. In the spirit of \rmref{recattolico1} and \eqref{reinacattolica1},
define first the map $\zeta_n: C^n(U,M) \to C^n(\tilde{U},\tilde{M})$ as follows: for $n=0$, define
$$
\zeta_0: M \longrightarrow \tilde{M}, \quad m \longmapsto \sum_j p'_j \otimes_\erre m \otimes_\erre q'_j,
$$
and for $n \geq 1$, abbreviating $y:= u^1 \otimes_\erre \cdots \otimes_\erre u^n$, define 
\begin{footnotesize}
\begin{equation*}
\begin{split}
\zeta_n: 
y \otimes_\erre m \longmapsto  \sum_{\scriptscriptstyle{i_{0, \ldots, {n-1}}} \atop \scriptscriptstyle{j_{0, \ldots, n}}} 
 & \big((p'_{j_0} \otimes q_{i_0}^o) \otimes_\Ree u^1 \otimes_\Ree (q'_{j_0} \otimes {p'_{j_1}}^o)\big) \otimes_\esse 
\big((p_{i_0} \otimes q_{i_1}^o) \otimes_\Ree u^2 \otimes_\Ree (q'_{j_1} \otimes {p'_{j_2}}^o)\big)  \\
&\quad \otimes_\esse  \cdots \otimes_\esse \big((p_{i_{n-2}} \otimes q_{i_{n-1}}^o) \otimes_\Ree u^{n} \otimes_\Ree (q'_{j_{n-1}} \otimes {p'_{j_{n}}}^o)\big) \otimes_\esse (p_{i_{n-1}} \otimes_\erre m \otimes_\erre q'_{j_n}).  
\end{split}
\end{equation*}
\end{footnotesize}
Second, define the map $\xi_n: C^n(\tilde{U},\tilde{M}) \to  C^n(U,M)$, which is
$$
\xi_0: \tilde{M} \longrightarrow M, \quad \big( \td{m}:=p \otimes_\erre m \otimes_\erre q\big) \longmapsto \sum_i(q_ip)m(qp_i),
$$
in degree $n = 0$, and for $n \geq 1$ is given by
\begin{equation*}
\begin{footnotesize}
{
\begin{split}
& \xi_n:  \td{y} \otimes_\esse \tilde{m} \longmapsto  
\sum_{i_{0, \ldots, n}} 
\lr{ (q_{i_0}a_1) \lact (q_{i_1}d_1) \blact u^1 \ract (b_1a_2) \bract (c_1p_{i_0}) } \otimes_\erre 
\lr{(q_{i_2}d_2) \blact u^2 \ract (b_2a_3) \bract (c_2p_{i_1}) } \otimes_\erre \cdots
 \\
&  \quad \otimes_\erre \lr{(q_{i_{n-1}}d_{n-1}) \blact u^{n-1} \ract (b_{n-1}a_n) \bract (c_{n-1}p_{i_{n-2}})} 
\otimes_\erre 
\lr{(q_{i_n} d_n) \blact u^n \ract (b_n p) \bract (c_np_{i_{n-1}})} \otimes_\erre m(qp_{i_n}),
\end{split} }
\end{footnotesize}
\end{equation*}
where $\tilde{u}^i := (a_i \otimes b^o_i) \otimes_\Ree u^i \otimes_\Ree (c_i \otimes d^o_i) \,\in \td{U}$ 
for $1 \leq i \leq n$, and
$\tilde{y}:= \tilde{u}^1 \otimes_\esse \cdots \otimes_\esse \tilde{u}^n$.

Third, introduce the homotopy $h_n: C^{n+1}(U, M) \to C^{n}(U,M)$ as follows: in degree $n=0$, set 
$$
h_0: u \otimes_\erre m \longmapsto \sum_{i,\,j} \gve\big((q_i p'_j) \blact u\big) m(q'_jp_i),
$$
and for $n\geq 1$ define
\begin{footnotesize}
\begin{equation}
\begin{split}
\label{amagica}
h_n: y' \otimes_\erre m  \longmapsto 
& \sum^n_{k=0}  \underset{\underset{i_{0, \ldots, k}}{j_{0,\ldots, k}}}{\sum}  (-1)^{k+n} 
u^0 \otimes_\erre \cdots \otimes_\erre u^{n-k-1} \otimes_\erre \gve\big((q_{i_0} p'_{j_0}) \blact u^{n-k}\big) \otimes_\erre \\
&  
\lr{ (q_{i_1}p'_{j_1}) \blact u^{n-k+1} \bract (q'_{j_0}p_{i_0}) }
\otimes_\erre   \cdots  \otimes_\erre
\lr{ (q_{i_k}p'_{j_k}) \blact u^{n} \bract (q'_{j_{k-1}}p_{i_{k-1}}) } \otimes_\erre m(q'_{j_k}p_{i_k}),
\end{split}
\end{equation}
\end{footnotesize}
abbreviating here $y':=u^0 \otimes_\erre \cdots \otimes_\erre u^n$.

Now, with the construction of $\td{U}$ and $\td{M}$ as in \S\ref{subsect:Morita} and analogously to Lemma \ref{Lemma:II}, one can construct a cocyclic module $C^\bull(\td{U}, \td{M})$; we leave the tedious details to the reader. 
Similarly as in Lemma \ref{revell}, Proposition \ref{prop: h}, and Lemma \ref{Lemma:tgamma}, one then proves:
\begin{lem}
The maps $\zeta_{\bull}$ and $\xi_{\bull}$ are morphisms of cochain complexes, and $\xi_\bull \zeta_\bull$ is homotopic to the identity by means of the homotopy \rmref{amagica}; likewise, $\zeta_\bull \xi_\bull$ is homotopic to the identity as well. In particular, $\zeta_\bull$ and $\xi_\bull$ are equivalences of cocyclic modules.
\end{lem}

This enables us to conclude:
\begin{thm}
\label{thm:main2}
{\rm (Morita base change invariance of (Hopf-)cyclic cohomology).}
Let $(R, U)$ be a left Hopf algebroid, $M$ a left $U$-comodule right $U$-module which is SaYD (i.e., satisfies \eqref{campanilla1}--\eqref{Eq:SaYD}), and $(R,S,P,Q,\phi,\psi)$ a Morita context. Then 
\begin{eqnarray*}
H^\bull(U, M) &\simeq& H^\bull(\td{U}, P \otimes_\erre M \otimes_\erre Q), \\
HC^\bull(U, M) &\simeq& HC^\bull(\td{U}, P \otimes_\erre M \otimes_\erre Q) 
\end{eqnarray*}
are isomorphisms of $\Bbbk$-modules.
\end{thm}

\begin{rem}
\label{cat}
The proofs of both Theorems \ref{thm:main2} \& \ref{thm:main1} are based on an explicit construction of (co)chain homotopies. One could wonder if a more categorical way implicitly leads to the same result but with less computational effort. 
Closest to our setting is perhaps \cite{BoeSte:CCOBAAVC}, where a categorical approach to the cyclic (co)homology of bialgebroids was developed based on the notion of admissible septuples.  
Our situation fits in the particular examples of admissible septuples of \cite[Propositions 1.15 \& 1.25]{BoeSte:CCOBAAVC}; nevertheless, none of the (co)cyclic objects that, after some additional steps, can be deduced from 
{\em loc.~cit.}
coincides with our (co)cyclic modules from \rmref{adualnightinpyongyang} and \rmref{anightinpyongyang}, respectively. Since also the involved tensor products differ, the two approaches are not even related by considering cyclic duals.
That is, our cyclic (co)homology seems to be different from the one considered in \cite{BoeSte:CCOBAAVC}.

Let us explain how far one can go in applying the approach of \cite{BoeSte:CCOBAAVC} to Morita base change invariance.
Following the notation of \cite{BoeSte:CCOBAAVC}, one can show that given a category $\cC$, two equivalent categories $\cM$, $\cN$, and an admissible septuple $\mathcal{S}=(\cM, \cC, \bT_l, \bT_r, \boldsymbol{\Pi}, \bt, \bi)$ over $\cM$, there is an admissible septuple $\td{\mathcal{S}}=(\cN, \cC,\td{ \bT}_l,\td{ \bT}_r, \td{\boldsymbol{\Pi}}, \td{\bt}, \td{\bi})$ over $\cN$, whose corresponding categories of transposition morphisms $\cW_{\cS}$ and $\cW_{\td{\cS}}$ are also equivalent. 
Under some natural assumption on $\cC$, we know from \cite[Corollary 1.11]{BoeSte:CCOBAAVC} that there is a functor 
$\hat{Z}^*(\cS, -): \cW_{\cS} \to  \boldsymbol{\Delta C}^{\,\cC}$ to the category of cocyclic objects of $\cC$. 
In this way, the Morita invariance theory in this context can be interpreted as follows. 
Consider a transposition morphism 
$(X, \omega) \in \cW_{\cS}$ and its image  $(\td{X},\td{\omega}) \in \cW_{\td{\cS}}$. One can then 
assign to them two cocyclic objects $\hat{Z}^*(\cS,
  (X,\omega))$ and $\hat{Z}^*(\td{\cS}, (\td{X},\td{\omega}))$ in $\cC$. 
Morita invariance now claims that the associated (co)chain complexes were quasi-isomorphic.
At this level of generality, there is seemingly no way which directly furnishes such a quasi-isomorphism if not, analogously to our approach, by manually constructing such a map in special cases (e.g., the aforementioned particular admissible septuples of \cite[Propositions 1.15 \& 1.25]{BoeSte:CCOBAAVC}). 
\end{rem}

\section{Applications and Examples.}
We give two applications. The first one deals with the well-known Morita invariance of the usual Hochschild and cyclic homology for associative algebras. 
We show that this invariance theory is a consequence of our main Theorem \ref{thm:main1} by applying it to the 
left Hopf algebroids $\Reh$ and $\Se$. 
In the second application we specialise our general results to the Morita context between the complex-valued smooth functions on the commutative real $2$-torus $\mathbb{T}^2:=\mathbb{S}^1 \times \mathbb{S}^1$ and the  coordinate ring of the noncommutative $2$-torus  with rational parameter, establishing thereby a passage from commutative to noncommutative geometry. We will first review  the construction of this context, and next  apply the Morita base change invariance of the cyclic homology between the left Hopf algebroid attached to the Lie algebroid of vector fields on $\mathbb{T}^2$, and the associated Morita base change left Hopf algebroid over this noncommutative $2$-torus whose structure maps are deduced from  \S\ref{subsect:Morita}. 

\subsection{Morita invariance of cyclic homology for associative algebras}\label{ssec:algebras} 
Recall from \cite{Schau:DADOQGHA} the left Hopf algebroid structure of the enveloping algebra $\Reh$. 
Its structure maps are given as follows: $s(r) := r \otimes 1$, $t(r^o) := 1 \otimes r^o$, 
$\gD(r \otimes \tilde{r}^o) := (r \otimes 1) \otimes_\erre (1 \otimes \tilde{r}^o)$,
$\varepsilon(r \otimes \tilde{r}^o) := r\tilde{r}$, and the inverse of the Hopf-Galois map is given as 
$
(r \otimes \td{r}^o)_+ \otimes_\Ropp (r \otimes \td{r}^o)_- := (r \otimes 1) \otimes_\Ropp (\td{r} \otimes 1).
$

Let now $M$ be a right $\Reh$-module which is also an $\Reh$-comodule with compatible left $R$-actions as in \rmref{campanilla1}, and denote the coaction by $m \mapsto (m'_{(-1)} \otimes m^{''}_{(-1)}) \otimes_\erre m_{(0)}$, omitting the summation symbol in all what follows.
Under the isomorphism 
$C_\bull(\Reh, M) = 
M \otimes_\Ropp \Reh^{\otimes_\Ropp
\bull} 
\simeq M \otimes R^{\otimes \bull}$ 
given by
\begin{equation}
\label{kababking}
m \otimes_\Ropp (r_1 \otimes \tilde{r}^o_1) \otimes_\Ropp \cdots \otimes_\Ropp (r_n \otimes \td{r}_n^o) 
\lma 
\td{r}_n \cdots \td{r}_1 m \otimes r_1 \otimes \cdots \otimes r_n,   
\end{equation}
the para-cyclic operators 
\rmref{adualnightinpyongyang} assume the form
\begin{equation}
\label{ceuta}
\!\!\!
\begin{array}{rcll}
d_i(m \otimes y)  &\!\!\!\!\! =& \!\!\!\!\!
\left\{ \!\!\!
\begin{array}{l}
r_n m \otimes r_1  \otimes   \cdots   \otimes   r_{n-1}
\\
m \otimes \cdots \otimes  r_{n-i} r_{n-i+1} \otimes  \cdots \otimes r_{n}
\\
m r_1 \otimes r_2  \otimes   \cdots    \otimes  
r_n 
\end{array}\right.  & \!\!\!\!\!\!\!\!   \begin{array}{l} \mbox{if} \ i \!=\! 0, \\ \mbox{if} \ 1
\!  \leq \! i \!\leq\! n-1, \\ \mbox{if} \ i \! = \! n, \end{array} \\
\
\\
s_i(m \otimes y) &\!\!\!\!\! =&\!\!\!\!\!  \left\{ \!\!\!
\begin{array}{l} m \otimes r_1 \otimes \cdots \otimes r_n \otimes 1
\\
m \otimes \cdots \otimes  r_{n-i} \otimes 1 \otimes r_{n-i+1} \otimes \cdots  \otimes r_{n}
\\
m \otimes 1 \otimes r_1 \otimes  \cdots \otimes  r_n 
\end{array}\right.   & \!\!\!\!\!\!\!  \begin{array}{l} 
\mbox{if} \ i\!=\!0, \\ 
\mbox{if} \ 1 \!\leq\! i \!\leq\! n-1, \\  \mbox{if} \ i\! = \!n, \end{array} \\
\
\\
t_n(m \otimes y) 
&\!\!\!\!\!=&\!\!\!\!\! 
m^{''}_{(-1)} m_{(0)} r_1 \otimes r_2 \otimes \cdots
\otimes r_n \otimes m'_{(-1)}, 
& \\
\end{array}
\end{equation}
where we abbreviate $y:=r_1
\otimes \cdots \otimes r_n$, and as before $C_\bull(\Reh,M)$ is
cyclic if $M$ is SaYD.

Using the isomorphism
\begin{equation}
\label{bancopopular}
\begin{array}{rcl}
\Pe \otimes_\Ree \Reh \otimes_\Ree \Qe &\overset{\simeq}{\longrightarrow}& \Se, \\
(a \otimes b^o) \otimes_\Ree (r \otimes \td{r}^o) \otimes_\Ree (c \otimes d^o)
&\longmapsto& \phi(a \otimes_\erre rc) \otimes \phi(d\td{r} \otimes_\erre b)^o, \\
\end{array}
\end{equation}
where $\phi$ is as in \eqref{Eq: Morita maps},
together with \rmref{energiaproxima} and 
the isomorphism $C_\bull(\Se, \td{M}) \simeq \td{M} \otimes S^{\otimes n}$ analogously to \rmref{kababking}, 
a straightforward computation reveals that the morphism of chain complexes \rmref{recattolico1} reads 
\begin{equation*}
\begin{split}
\theta_n: 
m \otimes y \longmapsto  \sum_{\scriptscriptstyle{i_{0, \ldots, n}}} 
 & (p_{i_0} \otimes_\erre m \otimes_\erre q_{i_1}) \otimes \phi(p_{i_1} \otimes_\erre r_1 q_{i_2}) \otimes \cdots \otimes \phi(p_{i_n} \otimes_\erre r_n q_{i_0}). 
\end{split}
\end{equation*}
In the other direction, we make use of the isomorphism 
\begin{eqnarray*}
\Qe \otimes_\See \Se \otimes_\See \Pe &\overset{\simeq}{\longrightarrow}& \Reh, \\
(c \otimes d^o) \otimes_\See (s \otimes \td{s}^o) \otimes_\See (a \otimes b^o)
&\longmapsto& \psi(c\td{s} \otimes_\esse a) \otimes \psi(b \otimes_\esse sd)^o,
\end{eqnarray*}
together with the inverse of \rmref{bancopopular} given by, cf.~\rmref{Eq:tilde_eta},
$$
\Se \lra \Pe \otimes_\Ree \Qe, \quad s \otimes \td{s}^o  \lma \sum_{i,\, j} (sp'_j \otimes {q'_i}^o) \otimes_\Ree (q'_j \otimes {(\td{s}p'_i)}^o),
$$
to conclude that the morphism of chain complexes \rmref{reinacattolica1} becomes here
\begin{equation*}
\begin{split}
\gamma_n: 
(p & \otimes_\erre m \otimes_\erre q) \otimes z \longmapsto \\
& \sum_{\scriptscriptstyle{j_{0, \ldots, n}}} 
  (\psi(q'_{j_0} \otimes_\esse p) m \psi(q \otimes_\esse p'_{j_1}) 
\otimes \psi(q'_{j_1} \otimes_\esse s_1 p'_{j_2}) 
\otimes \cdots \otimes \psi(q'_{j_n} \otimes_\esse s_n p'_{j_0}), 
\end{split}
\end{equation*}
abbreviating $z := s_1 \otimes \cdots \otimes s_n$. 

In a similar manner, one derives the homotopy \rmref{reyescattolicos} in this case: for $n=0$, we obtain
\begin{equation*}
\begin{split}
& h_0:  m  \longmapsto \sum_{i,j} m \psi(q_i\tensor{}p_j') \tensor{}  \psi(q_j'\tensor{}p_i) ,  
\end{split}
\end{equation*}
and for $n \geq 1$:
\begin{footnotesize}
\begin{multline*}
h_n: m \otimes y \longmapsto 
 \sum_{k=0}^{n}\,\,\underset{\underset{i_{0, \ldots,k}}{j_{0, \ldots, k}}}{\sum} (-1)^{n+k} 
(m \psi (q_{i_0}\tensor{} p'_{j_0}) )   \otimes \Big( \psi(q'_{j_0}\tensor{}p_{i_0})  r_1 \psi (q_{i_1}\tensor{}p'_{j_1})  \Big) 
 \otimes \cdots \\ \quad \otimes  \Big( \psi(q'_{j_{n-k-1}}\tensor{}p_{i_{n-k-1}})  r_k \psi (q_{i_{n-k}}\tensor{} p'_{j_{n-k}}) \Big) 
  \otimes  \psi(q'_{j_{n-k}}\tensor{} p_{i_{n-k}})   \otimes r_{k+1} \otimes \cdots \otimes r_n, 
\end{multline*}
\end{footnotesize}
where we abbreviate $y:=r_1
\otimes \cdots \otimes r_n$.

One recovers the explicit maps given in \cite{McC:MEACH} for this situation, and hence from Theorem \ref{thm:main1} the classical result \cite{McC:MEACH, DenIgu:HHATSOFP} 
of Morita invariance in Hochschild theory follows. In case $M = R$, with \cite[Prop.~3.1]{KowPos:TCTOHA}, one furthermore 
reproduces the classical result of Morita invariance of cyclic homology of associative algebras from \cite{Con:NCDG, LodQui:CHATLAHOM,McC:MEACH}:
\begin{cor}
\label{cor:alg}
Let $R$ be an associative $\Bbbk$-algebra, 
$M$ an $(R,R)$-bimodule, and $(R,S,P,Q,\phi,\psi)$ a Morita context. 
We then have the following natural $\Bbbk$-module isomorphism
$$
H^{\rm alg}_\bull(R, M) \simeq H^{\rm alg}_\bull(S, P \otimes_\erre M \otimes_\erre Q), 
$$ 
and in case $M:=R$, we obtain
\begin{equation}
\label{seitenbachervollkornmuesli}
HC^{\rm alg}_\bull(R) \simeq HC^{\rm alg}_\bull(S). 
\end{equation}
\end{cor}
Observe that for this corollary no SaYD condition is needed: there is no coaction required to compute the homology of the underlying simplicial object in \rmref{ceuta} (resp.~\rmref{adualnightinpyongyang}), and for the cyclic homology in \rmref{seitenbachervollkornmuesli} we only considered the case $M:=R$, with action given by multiplication and coaction $R \to \Reh \otimes_\erre R \simeq \Reh, \ r \mapsto r \otimes_k 1$, which is easily seen to define an SaYD module.

\subsection{Morita base change invariance in Lie algebroid theory and the noncommutative torus} 
\subsubsection{Lie algebroids and associated left Hopf algebroids}
\label{ssection: Lie alg}
Assume that $R$ is a commutative $\Bbbk$-algebra (here $\Bbbk$ is a ground field of characteristic zero) 
and denote by $\mathrm{Der}_{\Bbbk}(R)$ the Lie algebra of all $\Bbbk$-linear derivations of $R$. 
Consider a $\Bbbk$-Lie algebra $L$ which is also an $R$-module, 
and let $\omega: L \to \mathrm{Der}_{\Bbbk}(R)$ be a morphism of $\Bbbk$-Lie algebras. 
Following \cite{Rin:DFOGCA}, the pair $(R,L)$ is called \emph{Lie-Rinehart algebra} with \emph{anchor} map $\omega$, provided
\begin{eqnarray*}
 (aX) (b) &=&  a(X(b)), \qquad  \\
{[ X, aY ]} &=& a{[X,Y]}+X(a)Y,
\end{eqnarray*}
for all $X, Y \in L$ and $a, b \in R$, where $X(a)$ stands for $\omega(X)(a)$.
A \emph{morphism $(R,L) \to (R,L')$ of Lie-Rinehart algebras over $R$} is a map $\varphi: L \to L'$ of  $\Bbbk$-Lie algebras such that 
$$
\xymatrix{ L \ar@{->}_-{\omega}[rd]  \ar@{->}^-{\varphi}[rr] & & L' \ar@{->}^-{\omega'}[ld] \\ & \mathrm{Der}_{\Bbbk}(R) & }
$$
is a commutative diagram. These objects form a category which we denote by $\mathbf{LieRine}_{(\Bbbk,\,R)}$.

\begin{example}
\label{Exam: bundles}
Here are the basic examples which we will be dealing with, and which motivate the above general definition:
\begin{enumerate}
\item The pair $(R, \mathrm{Der}_{\Bbbk}(R))$ trivially admits the structure of a Lie-Rinehart algebra. 
\item A \emph{Lie algebroid}  is a  vector bundle $\mathcal{E} \to \mathcal{M}$ over a smooth manifold, together with a map $\omega: \mathcal{E} \to T\mathcal{M}$ of vector bundles and a Lie structure $[-,-]$ on the  vector space  $\Gamma(\cE)$  of global smooth sections of $\mathcal{E}$, such that the induced map $\Gamma(\omega): \Gamma(\cE) \to \Gamma(T\mathcal{M})$ is a Lie algebra homomorphism, and for all $X, Y \in \Gamma(\cE)$ and any $f \in \mathcal{C}^{\infty}(\mathcal{M})$  one has $[X,fY]\,=\, f[X,Y]+ \Gamma(\omega)(X)(f)Y$.  Then the pair $(\mathcal{C}^{\infty}(\mathcal{M}), \Gamma(\cE))$ is obviously a Lie-Rinehart algebra.
\end{enumerate} 
\end{example}

Associated to any Lie-Rinehart algebra $(R, L)$ there is a universal object denoted by $(R,\cV L)$, see \cite{Rin:DFOGCA, Hue:PCAQ}. 
Using the notion of \emph{smash product} (or, more general, of \emph{distributive law between two algebras}), we give here an alternative construction (of which the {\em Massey-Peterson algebra} in \cite{MasPet:TCSOCFS,Hue:PCAQ} is a special case) of this object:  
let $UL$ be the universal enveloping algebra of $L$ with its canonical Hopf algebra structure, and consider the $\Bbbk$-linear map $L \overset{\omega}{\longrightarrow} 
\mathrm{Der}_{\Bbbk}(R)$. 
Extending this map to $UL$, we obtain the structure of a $UL$-module algebra on $R$. Following \cite[pp.~117--118]{Swe:GOSA}, the smash product $R\# UL$ admits the structure of a left $R$-bialgebroid, where the source and the target map coincide.  
Now take the following factor $R$-algebra of $R\# UL$:
$$
\pi: R\# UL \longrightarrow \cV L:=\,\, \frac{R\# UL}{\cJ_L},
$$
where
$$    
\cJ_L: ={\langle  a \# X - 1\# aX\rangle}_{a \in R,\, X \in L}
$$ 
is the two sided ideal generated by the set $\{a \# X - 1\# aX\}_{a \in R,\, X \in L}$. 
The $R$-bialgebroid structure of  $R\# UL$ projects to $\cV L$ (see \cite{Xu:QG}), and by  \cite[\S4.2.1]{Kow:HAATCT} one also has that $\cV L$ carries a left Hopf algebroid structure, the 
translation map on generators $a \in R$, $X \in L$ of $\cV L$ being given by
$$
a_{+}\tensor{\Ropp} a_{-} := a\tensor{\Ropp} 1, \qquad  X_{+}\tensor{{\Ropp}}X_{-}:= X\tensor{\Ropp}1 - 1 \tensor{\Ropp} X,
$$ 
where $\cV L \otimes_\Ropp \cV L := \due {\cV L} \blact {} \otimes_\Ropp \due {\cV L} {} \ract$  is as in \rmref{tata} (which is why we stick to the symbol $\Rop$ although $R$ is commutative), and where we identify the elements of $R$ and $L$ with their respective images by the universal  maps $\iota_R: R \to \cV L$, $a \mapsto a\# 1 + \cJ_L$ 
and $\iota_L: L \to \cV L$, $X \mapsto 1\# X +\cJ_L$.

\subsubsection{Vector bundles versus  $\sqrt{\mbox{Morita}}$ theories} \label{ssection: versus} 
Let $R$ be a commutative  $\Bbbk$-algebra as in \S\ref{ssection: Lie alg}.  
Assume we are given  a finitely generated and projective module $P_R$ which is faithful over $R$.
Then it is well known (see, for example, \cite[Corollary 1.10]{DeMeyer/Ingraham:1971}) that $R$ is Morita equivalent to the endomorphism 
ring $\End{P_R}$ since $R$ is commutative.  The context maps are given by 
$$
\begin{array}{rrrllll}
\phi:& P\tensor{R}P^* & \overset{\simeq}{\longrightarrow}& \End{P_R}, & (p\tensor{}\sigma & \longmapsto & [ p' \mapsto p\sigma(p')]), \\  
\psi:& P^*\tensor{\End{P_R}} P & \overset{\simeq}{\longrightarrow}& R, & (\sigma\tensor{}p & \longmapsto & \sigma(p)), 
\end{array}
$$
where $P^*={\rm Hom}(P_R, R_R)$.

Following \cite[Example 2.3.3]{Khalkhali:book}, we apply this Morita context  to the situation where $R$ is the algebra of smooth functions over a manifold $\cM$.
By the Serre-Swan theorem, it is well known that for a (complex) smooth vector bundle $\pi: \cP \to \cM$ of constant rank $\geq 1$ 
the global smooth sections $P:=\Gamma(\cP)$ form a finitely generated projective module over the commutative ring $R:=\cC^{\infty}(\cM)$ of complex-valued smooth functions on $\cM$, see, for instance, \cite[Remark, p.~183]{Nestruev:03}. One can furthermore show \cite[Remarque 2, p.~145]{Bou:AC12} that $P$ is of constant rank $\geq 1$ (the rank of $\pi$), and as such,  $P$ becomes a faithful $R$-module (as follows from \cite[p.~142, Corollaire, \& p.~143, Th\'eor\`eme 3 (ii)]{Bou:AC12}).
Therefore, $\cC^{\infty}(\cM)$ is Morita equivalent to the endomorphism algebra $\End{P_{\cC^{\infty}(\cM)}} \simeq \Gamma(\End{\cP})$. 
In this way, there is a functor from the category $\mathbf{LieRine}_{(\mathbb{C},\, \cC^{\infty}(\cM))}$ 
to the category of left Hopf algebroids over $\End{P_{\cC^{\infty}(\cM)}}$. 
This functor is defined on objects by sending any complex Lie-Rinehart algebra $(L,R)$ to the left $\End{P_R}$-Hopf algebroid $\Pe\tensor{\Reh}\cV L\tensor{\Reh}\Qe$, where $\Pe$, $\Qe$ are defined as in \S\ref{ssec: context} and correspond to the Morita context $(R, \End{P_R},P, P^*, \phi,\psi)$, i.e., with $Q\,=\, P^*$.

\begin{rem}\label{remark:field extension problem}
An analogue to the previous functor can, in fact, descend to the category of Lie algebroids 
over a smooth manifold $\cM$ if we take a real vector bundle and the algebra $\cC^{\infty}(\cM,\, \mathbb{R})$  
of real-valued smooth functions instead of $\cC^\infty(\cM) = \cC^{\infty}(\cM,\, \mathbb{C})$. 
We know from  Example \ref{Exam: bundles}(ii) that there is a  canonical faithful functor 
 from the category of Lie algebroids over $\cM$ to the category of real Lie-Rinehart algebras over $\cC^{\infty}(\cM,\, \mathbb{R})$. 
Now we can compose this functor with the one constructed by the same process as in \S\ref{ssection: versus}.  
In general, there is no obvious functor connecting the categories  $\mathbf{LieRine}_{(\mathbb{R},\, \cC^
 {\infty}(\cM,\, \mathbb{R}))}$ and  $\mathbf{LieRine}_{(\mathbb{C},\, \cC^{\infty}(\cM,\, \mathbb{C}))}$, except perhaps when $\cM$ is an almost complex manifold (i.e., a smooth manifold with a smooth endomorphism field $J: T\cM \to T\cM$ satisfying $J_x^2=-\id_{T\cM_x}$ for all $x\in \cM$).

Let us mention that due to our interest in the noncommutative torus,  we  have been forced to extend  the base field by using the complex-valued functions instead of real-valued ones. 
\end{rem}

The material of the following subsection will appear well known to the reader who is familiar with noncommutative differential geometry techniques. 
For the convenience of the rest of the audience,  we include a detailed exposition following ideas from \cite[\S3.1]{DKMM:2001}, \cite[\S1.1]{Khalkhali:book}.

\subsubsection{Noncommutative torus revisited.}\label{ssubsec:Tq}
Consider the Lie group $\mathbb{S}^1=\{z \in \mathbb{C}\setminus \{0\} \mid |z|= 1\}$ as a real $1$-dimensional torus by identifying it with the additive quotient $\mathbb{R}/2\pi \mathbb{Z}$. 
Likewise, the real $d$-dimensional torus $\mathbb{T}^d := \mathbb{S}^1 \times \cdots \times \mathbb{S}^1$ is identified with $\mathbb{R}^d/2\pi\mathbb{Z}^d$. The complex algebra of all smooth complex-valued functions on $\mathbb{T}^2$ will be denoted by $\cC^{\infty}(\mathbb{T}^2)$.

Fix a root of unity $\Sf{q} \in \mathbb{S}^1$ 
and take $N \in \mathbb{N}$ to be  the smallest natural number such that $\Sf{q}^N\,=\, 1$.  Let us consider the semidirect product group $\cG:=\mathbb{Z}_N^2 \ltimes \mathbb{S}^1$ where $\mathbb{Z}_N\,=\, \mathbb{Z}/N\mathbb{Z}$,
and operation
$$
(m,n,\theta) (m',n',\theta') := (m+m',n+n',\theta\theta'\Sf{q}^{mn'}),$$ for every pair of elements $(m,n,\theta),  \, (m',n',\theta') \in \cG$.
There is a right  action of the group $\cG$ on the torus $\mathbb{T}^3$ given as follows:
$$
 (\Sf{x},\Sf{y},\Sf{z}) (m,n,\theta) := 
(\Sf{q}^m\Sf{x}, \Sf{q}^n \Sf{y}, \theta \Sf{z}\Sf{y}^m), \quad (\Sf{x}, \Sf{y},\Sf{z}) \in \mathbb{T}^3, \quad (m,n,\theta) \in \cG.
$$
Now, we can show that the map 
$$
{\bf p} : \mathbb{T}^3 \longrightarrow \mathbb{T}^2, \quad  (\Sf{x},\Sf{y},\Sf{z}) \longmapsto (\Sf{x}^N,\Sf{y}^N) 
$$
satisfies:
\begin{enumerate}
\item $\bf p$ is a surjective submersion;
\item $\cG$ acts freely on $\mathbb{T}^3$ and the orbits of this action coincide with the fibres of $\bf p$.
\end{enumerate}
As a consequence and by applying \cite[Lemma 10.3]{Kolar/Michor/Slovak:1993}, we see that 
$(\mathbb{T}^3,\mathbf{p},\mathbb{T}^2,\cG)$ is a principal fibre bundle.  
We then 
want to associate a non-trivial vector bundle to the trivial bundle  $\mathbb{T}^3 \times \mathbb{C}^N \to \mathbb{T}^3$.  So, we need to extend the $\cG$-action on   $\mathbb{T}^3$ to $\mathbb{T}^3 \times \mathbb{C}^N$, which 
is possible  by considering the following  left $\cG$-action on $\mathbb{C}^N$
$$
 \cG \longrightarrow  \mathrm{End}_{\mathbb{C}}(\mathbb{C}^N), \quad
(m,n,\theta) \longmapsto \big\{ \omega \longmapsto \theta U_0^nV_0^{-m}\omega \big\},
$$
where $U_0,V_0$ are the $(N\times N)$-matrices
$$
U_0\,=\, \begin{pmatrix} 0 & 1& 0 & 0 & \cdots \\ 0 & 0 & 1 & 0 & \cdots \\ \vdots &  & \ddots & \ddots & \vdots \\ 0 & \cdots &  & 0 & 1 \\ 1 & 0 & \cdots & 0  & 0 \end{pmatrix},
\qquad 
V_0\,=\, \begin{pmatrix} 1 & 0& \cdots & \cdots & 0 \\ 0 & \Sf{q} & 0 & \cdots & 0 \\  0 & 0 &   \sf{q}^{2} & \cdots & 0 \\  \vdots &  & \ddots & \ddots & \vdots \\ 0 & 0 & \cdots & 0  & \sf{q}^{N-1} \end{pmatrix},
$$
which satisfy the relations
\begin{equation}\label{Eq:quantum}
U_0V_0\,\,=\,\, \Sf{q} V_0U_0,\qquad U_0^N\,=\, V_0^N\,=\, \mathbb{I}_N.
\end{equation}
Therefore, we have  a right $\cG$-action on $\mathbb{T}^3\times \mathbb{C}^N$ defined by
$$
 \big( (\Sf{x},\Sf{y},\Sf{z}); \omega\big)\, (m,n,\theta) := \Big( (\Sf{x},  \Sf{y},  \Sf{z})(m,n,\theta);\,\, (m,n,\theta)^{-1} \omega  \Big)= 
\Big( (\Sf{q}^m\Sf{x}, \Sf{q}^n \Sf{y}, \theta \Sf{z}\Sf{y}^m);\,\, \theta^{-1}U_0^{-n}V_0^{m}\omega \Big).
$$
The orbit space $(\mathbb{T}^3\times \mathbb{C}^N)/\cG = \mathbb{T}^3\times_{\cG} \mathbb{C}^N$ with elements $u\times_{\cG}\omega$ will be denoted by $\cE_{\Sf{q}}$. Notice that by definition one has the following formula:
$$
(ug) \times_{\cG}\omega \,\,=\,\, u\times_{\cG} (g \omega), \text{ for every } u \in \mathbb{T}^3,\, \omega \in \mathbb{C}^N
, \text{ and } g \in \cG.
$$
By applying \cite[Theorem 10.7, \S10.11]{Kolar/Michor/Slovak:1993} we can associate a 
non-trivial vector bundle to the trivial bundle $\mathbb{T}^3 \times \mathbb{C}^N \to \mathbb{T}^3$, that is, 
there is a morphism of vector bundles
\begin{equation}\label{Eq:Eqp}
\xymatrix{  \mathbb{T}^3 \times \mathbb{C}^N  \ar@{->}[rr] \ar@{->}_{pr_1}[d] & &  \cE_{\Sf{q}}  
\ar@{-->}^-{\bara{\bf p}}[d]  \\ \mathbb{T}^3 \ar@{->}_-{{\bf p}}[rr] & & \mathbb{T}^2.}
\end{equation}
By the results of \S\ref{ssection: versus}, we have that $\cC^{\infty}(\mathbb{T}^2)$ 
is Morita equivalent to 
$\End{\Gamma(\cE_{\Sf{q}})} \simeq \Gamma\big( \End{\cE_{\Sf{q}}} \big)$. 
Now, using \cite[Theorem 10.12]{Kolar/Michor/Slovak:1993}, 
$\Gamma(\cE_{\Sf{q}})$ is identified with the $\cG$-equivariant subspace  $\cC^{\infty}(\mathbb{T}^3,\, \mathbb{C}^N)^{\cG}$ of 
$\cC^{\infty}(\mathbb{T}^3,\, \mathbb{C}^N)$, that is, those $f \in \cC^{\infty}(\mathbb{T}^3,\, \mathbb{C}^N)$ for which 
$f(ug)\,=\, g^{-1} f(u) $, for every  $u \in \mathbb{T}^3$, $g \in \cG$.
Hence, we have an isomorphism  $\Gamma(\cE_{\Sf{q}})\, \simeq \, \cC^{\infty}(\mathbb{T}^3,\, \mathbb{C}^N)^{\cG}$ of $\cC^{\infty}(\mathbb{T}^2)$-modules.

Next, we want to describe the noncommutative complex algebra $\End{\Gamma(\cE_{\Sf{q}})} \simeq \Gamma\big( \End{\cE_{\Sf{q}}} \big)$. Observe 
that there is a left $\mathbb{Z}_N^2$-action on  the $(N\times N)$-matrix algebra ${\rm M}_N(\mathbb{C})$ with complex entries, defined by 
\begin{equation}\label{Eq: matrix action}
(m,n) A := U_0^nV_0^{-m}AV_0^{m}U_0^{-n},\;\; \text{ for every  } A \in {\rm M}_N(\mathbb{C}), \, (m,n ) \in \mathbb{Z}_N^2.
\end{equation}
There is also a free right $\mathbb{Z}_N^2$-action on $\mathbb{T}^2$ given by
$$
(\Sf{x},\Sf{y}) (m,n) := (\Sf{q}^m\Sf{x},\Sf{q}^n\Sf{y}), \;\; \text{ for every } (\Sf{x},\Sf{y}) \in \mathbb{T}^2, \, (m,n) \in \mathbb{Z}_N^2.
$$
As before, one can construct 
the orbit space $\mathbb{T}^2 \times_{\mathbb{Z}_N^2} {\rm M}_N(\mathbb{C})$ after extending these actions to the trivial algebra bundle $\mathbb{T}^2 \times {\rm M}_N(\mathbb{C})$.    It turns out that the endomorphism algebra bundle $\End{\cE_{\Sf{q}}}$ is isomorphic 
to this orbit space, and clearly $\Gamma(\End{\cE_{\Sf{q}}})$ consists  of $\mathbb{Z}_N^2$-equivariant sections, that is,
\begin{equation}\label{Eq:T}
\Sf{T} \in \Gamma(\End{\cE_{\Sf{q}}})\;\; \text{ if and only if }\,\, \Sf{T}(\Sf{q}^m\Sf{x},\Sf{q}^n\Sf{y}) \,\,=\,\,(m,n) \Sf{T}(\Sf{x}, \Sf{y}),\; 
\end{equation}
for every $(\Sf{x}, \Sf{y}) \in \mathbb{T}^2$ and $ (m,n) \in \mathbb{Z}_N^2$, where on the right hand side we mean the action \eqref{Eq: matrix action}.

On the other hand, it is well known that $\cC^{\infty}(\mathbb{T}^2)$ 
can be identified with the algebra of all smooth functions on $\mathbb{R}^2$ that are $2\pi$-periodic w.r.t.\ each of their arguments. By Fourier expansion  $\cC^{\infty}(\mathbb{T}^2)$ consists of all functions
$$
f \,\,=\,\, \sum_{(k,\,l) \, \in \, \mathbb{Z}^2} f_{k,l} u^kv^l,
$$
where $\{f_{k, l}\}_{(k,\,l) \, \in \mathbb{Z}^2}$ is any rapidly decreasing sequence of complex numbers, that is, for every $r \in \mathbb{N}$, the seminorm
\begin{equation}\label{Eq:Frechet}
\| f\|_r\,\,=\,\, \underset{k,\, l\, \in\mathbb{Z}}{\mathrm{sup}}\Big( |f_{k, l}| (1+ |k|+|l|)^r\Big) \, < \infty,
\end{equation}
and where $u=e^{i2\pi t}$, $v=e^{i2\pi s}$ are the coordinate functions on the torus $\mathbb{T}^2$.

It is also well known that the complex matrix algebra ${\rm M}_N(\mathbb{C})$ is generated as $\mathbb{C}$-algebra by the elements $U_0, V_0$. 
Thus, Eqs.~\eqref{Eq:quantum} and \eqref{Eq:T} force any $\Sf{T} \in \Gamma(\End{\cE_{\Sf{q}}})$ to be of the form 
$$ 
\Sf{T} \,\, =\,\, \sum_{k,l,\, \in \mathbb{Z}} \Sf{T}_{k,\, l}(uU_0)^k(vV_0)^l ,
$$ 
with coefficients $\{  \Sf{T}_{k,\, l}\}_{\mathbb{Z}^2}$ satisfying Eq.~\eqref{Eq:Frechet}. 
Therefore, there is now a $\mathbb{C}$-algebra isomorphism
\begin{equation*}
\label{Eq:nocom torus}
\Gamma(\End{\cE_{\Sf{q}}}) \to  \cC^{\infty}(\mathbb{T}^2_{\Sf{q}}),  \quad 
\big( (uU_0) \mapsto  U, \quad (vV_0) \mapsto V\big),
\end{equation*}
where $\cC^{\infty}(\mathbb{T}^2_{\Sf{q}})$ refers to the complex noncommutative $2$-torus whose elements are formal power Laurent series in $U, V$ with a rapidly decreasing sequence of coefficients (cf.~\eqref{Eq:Frechet}), subject to $UV = \Sf{q} VU$. 
In conclusion, we have the 
Morita context $(\cC^{\infty}(\mathbb{T}^2), \cC^{\infty}(\mathbb{T}_{\Sf{q}}^2), \Gamma(\cE_{\Sf{q}}), \Gamma(\cE_{\Sf{q}})^*)$, where in addition 
$\cC^{\infty}(\mathbb{T}^2)$ and $\cC^{\infty}(\mathbb{T}_{\Sf{q}}^2)$ are related by the algebra map
$$
\cC^{\infty}(\mathbb{T}^2) \longrightarrow  \cC^{\infty}(\mathbb{T}_{\Sf{q}}^2), \quad( u \mapsto U^N, \,\, v \mapsto V^N).
$$

In the next subsection, we will use the Morita context stated above 
together with Theorems \ref{thm:main1}  and \ref{thm:main2} to prove the Morita invariance of both cyclic homology and  cohomology from the left Hopf algebroid attached to the Lie algebroid of vector fields over the classical $2$-torus to the associated Morita base change left Hopf algebroid over the noncommutative $2$-torus (the primitive elements of which can be seen to consist of noncommutative vector fields, cf.\ the comment in the Introduction), using the construction performed in  \S\ref{subsect:Morita}.

\subsubsection{The cyclic homology  for the left Hopf algebroid over  the noncommutative torus}\label{ssubsec:q} 
Now we will direct our attention to the Morita invariance of the cyclic homology 
between the trivial Lie algebroid $\big( \cC^{\infty}(\mathbb{T}^2),\, K:={\rm Der}_{\mathbb{C}}\big( \cC^{\infty}(\mathbb{T}^2)\big)\big)$ and its induced 
left Hopf algebroid $(S, \widetilde{\cV K}:=\Pe\tensor{\Reh}\cV K \tensor{\Reh} \Qe)$, where 
\begin{equation}\label{Eq: Morita torus}
R:=\cC^{\infty}(\mathbb{T}^2), \;\;  S:= \cC^{\infty}(\mathbb{T}_{\Sf{q}}^2),\;\;  P:=\Gamma(\cE_{\Sf{q}}),\;\; Q:=  \Gamma(\cE_{\Sf{q}})^{*},
\end{equation}
and where the notation is that of \S\ref{ssubsec:Tq}. Here we can explicitly compute the structure maps of the left Hopf algebroid $\widetilde{\cV K}$ by using  the general description of \S\ref{subsect:Morita}, as well as the dual basis of $P$ which can be extracted from the dual basis of the trivial bundle $\mathbb{T}^3 \times \mathbb{C}^N$, see Eq. \eqref{Eq:Eqp}.  Applying  Theorems \ref{thm:main1} \& \ref{thm:main2} as well as \cite[Theorem 5.2]{KowKra:CSIACT} (and its dual version, cf.~\cite[Theorem~3.14]{KowPos:TCTOHA}), we obtain
\begin{cor}
Let $\Sf{q} \in \mathbb{S}^1$ be a root of unity,
and consider the Lie algebroid $(R,K)$  of vector fields of the complex torus $\mathbb{T}^2$ and its induced left Hopf algebroid $(R,\cV K)$. 
Let $M$ be a right $\cV K$-module 
and $(R,S,P,Q,\phi,\psi)$ the Morita context of Eq.~\eqref{Eq: Morita torus}. 
We then have the following natural $\mathbb{C}$-module isomorphisms
$$
\begin{array}{rclcrcl}
H_\bull(\cV K , M) &\simeq & H_\bull( \widetilde{\cV K}, \td{M}), && HC_\bull(\cV K, M) &\simeq& HC_\bull( \widetilde{\cV K}, \td{M}), \\
H^\bull(\cV K , M) &\simeq & H^\bull( \widetilde{\cV K}, \td{M}), && HC^\bull(\cV K, M) &\simeq& HC^\bull( \widetilde{\cV K}, \td{M}), 
\end{array}
$$
where $\widetilde{\cV K}$ is the Morita base change left Hopf algebroid over the noncommutative torus $\cC^{\infty}(\mathbb{T}_{\Sf{q}}^2)$.\\
Furthermore,  assume that $M$ be $R$-flat. Then we have that
$$
\begin{array}{rclcrcl}
H_\bull( \widetilde{\cV K}, \td{M}) &\!\!\!\!\simeq&\!\!\!\!  H_\bull(K, M), && 
HC_\bull( \widetilde{\cV K}, \td{M}) &\!\!\!\!\simeq&\!\!\!\! \textstyle\bigoplus_{i \geq 0}H_{\bull -2i}(K,M), \\
H^\bull( \widetilde{\cV K}, \td{M}) &\!\!\!\!\simeq&\!\!\!\!  M \otimes_\erre \textstyle\bigwedge^\bull_\erre K, && 
HP^\bull( \widetilde{\cV K}, \td{M}) &\!\!\!\!\simeq&\!\!\!\! \textstyle\bigoplus_{i \equiv \bull {\rm mod} 2}H_{i}(K,M) \\
\end{array}
$$ 
are natural $\mathbb{C}$-module isomorphisms,
where $H_\bull(K, M) :=  \Tor^{\cV K}_\bull(M, R)$,
 and where $HP^\bull$ denotes periodic cyclic cohomology (see, e.g., \cite[\S5.1.3]{Lod:CH} for the definition of $HP$).
\end{cor} 

\begin{rem}
In \cite[Theorem 5.2]{Zharinov:2005}, the Hochschild cohomology of the algebra $R=\cC^{\infty}(\mathbb{T}^2)$ was computed in terms of the exterior algebra of a two-dimensional complex vector space. 
So we can apply Corollary \ref{cor:alg} to deduce the Hochschild cohomology of the noncommutative torus  $\cC^{\infty}(\mathbb{T}_{\Sf{q}}^2)$, where $q$ is not a root of unity.  On the other hand, the same result \cite[Theorem 5.2]{Zharinov:2005} shows that  $K={\rm Der}_{\mathbb{C}}\big( \cC^{\infty}(\mathbb{T}^2)\big)$ 
is a  free $R$-module of  rank $2$. One can therefore also consider another application of Theorems \ref{thm:main1} \& \ref{thm:main2} by taking  the canonical Morita context $(R, M_4(R),K, K^*, db, ev)$ and the left Hopf algebroid $(R,\cV K)$. Here $M_4(R)$ denotes the $(4 \times 4)$-matrices over $R$, whereas $ev : K^*\tensor{M_4(R)}K \to R, \varphi\tensor{M_4(R)}x \mapsto \varphi(x)$ stands for the evaluation map and $db : K\tensor{R}K^*  \to \rm{End}(K_R) \cong M_4(R) $ for the dual basis map which sends any element $x\tensor{R}  \varphi$ to the $(4\times 4)$-matrix attached to the $R$-linear map $\left[ y \mapsto x \varphi(y)\right]$.  The details of this application are left to the reader. 
\end{rem}

\bibliographystyle{amsplain}

\providecommand{\bysame}{\leavevmode\hbox to3em{\hrulefill}\thinspace}
\providecommand{\MR}{\relax\ifhmode\unskip\space\fi MR }
\providecommand{\MRhref}[2]{%
  \href{http://www.ams.org/mathscinet-getitem?mr=#1}{#2}
}
\providecommand{\href}[2]{#2}

\end{document}